\documentclass[11pt,a4paper]{amsart}
\usepackage[T1]{fontenc}
\usepackage{lmodern}
\usepackage{amssymb,mathtools}
\usepackage[margin=1in]{geometry}
\usepackage{graphicx}
\usepackage{subcaption}
\usepackage{paralist,nicefrac}
\usepackage[numbers,sort&compress]{natbib}
\usepackage[hidelinks]{hyperref}
\usepackage{cleveref}
\newtheorem{theorem}{Theorem}
\newtheorem{lemma}{Lemma}
\newtheorem{proposition}{Proposition}
\newtheorem{corollary}{Corollary}
\theoremstyle{remark}
\newtheorem{remark}{Remark}

\newcommand{\EE}{\mathsf{E}}
\newcommand{\PP}{\mathsf{P}}
\title[DAG-type distributed ledgers via young-age PA]{Directed acyclic graph-type distributed ledgers via young-age preferential attachment}
\author{Christian M\"onch}
\address{Institute of Mathematics, Johannes Gutenberg University, Mainz, Germany}
\email{cmoench@uni-mainz.de}
\author{Amr Rizk}
\address{Institute for Computer Science and Business Information Systems, University of Duisburg-Essen, Essen, Germany}
\email{amr.rizk@uni-due.de}
\thanks{The research of the first author is supported by the Deutsche Forschungsgemeinschaft (DFG, German Research Foundation)--443916008.}
\thanks{Revised author manuscript of the article published in \emph{Stochastic Systems} \textbf{13}(3) (2023), 377--397, \url{https://doi.org/10.1287/stsy.2022.0005}. The offspring-type law in Remark~\ref{rmk:locallimit} and the proof of Proposition~\ref{prop:nogiant} have been corrected.}
\date{November 24, 2026}
\begin{document}
\begin{abstract}
Distributed Ledger Technologies provide a mechanism to achieve ordering among transactions that are scattered on multiple  participants with no prerequisite trust relations.
This mechanism is essentially based on the idea of new transactions referencing older ones in a chain structure.
Recently, DAG-type Distributed Ledgers that are based on directed acyclic graphs (DAGs) were proposed to increase the system  scalability through sacrificing the total order of transactions.
In this paper, we develop a mathematical model to study the process that governs the addition of new transactions to the DAG-type Distributed Ledger.
We propose a simple model for DAG-type Distributed Ledgers that are obtained from a recursive Young-age Preferential Attachment scheme, i.e.\ new connections are made preferably to transactions that have not been in the system for very long. We determine the asymptotic degree structure of the resulting graph and show that a forward component of linear size arises if the edge density is chosen sufficiently large in relation to the `young-age preference' that tunes how quickly old transactions become unattractive.
\end{abstract}
\maketitle

\section{Motivation and background}\label{sec:intro}
\vspace{10pt}

\subsection{Introduction}\label{sec:intro_1}

Distributed Ledger Technologies (DLTs) which are based on directed acyclic graphs (DAG) such as the Tangle \cite{popov2016tangle} or Coordicide \cite{Coordicide2020} generalize Blockchain-based DLTs by weakening the ordering among the transaction blocks that constitute the ledger. Blockchain-based DLTs achieve total ordering of the transactions contained by the ledger through the rule of the longest-chain as coined in \cite{nakamoto2006bitcoin}. Here, transactions are collected into blocks that each include a reference to the previous block such that a chain of ordered transactions emerges. When blocks are created concurrently such that only a partial order exists the rule of the longest-chain dictates that blocks that do not belong to the longest chain are discarded. This is a major factor that impairs the scalability of Blockchain-based DLTs in terms of the rate of adding transactions to the ledger.

\medskip

DAG-type Distributed ledgers achieve scalability through allowing partially ordered transaction blocks, i.e.\ transaction blocks may reference more than one previous blocks. 
In this paper, we study the relation of design parameters for DAG-type DLTs based on the transient and limiting properties of the underlying transaction graph. We consider the graph of transaction blocks that emerges through block attachment, i.e. new blocks referencing preceding ones. We focus on two design choices for DAG-type DLTs, namely
\begin{itemize}
	\item \emph{How many preceding blocks should a new block reference?}
	\item \emph{How should the age of a preceding block impact the reference choice of a new block?}
\end{itemize}
The rationale behind these two choices lies in the formulation of simple, probabilistic block reference rules that guarantee the existence of a giant forward component based on \emph{(i)} a simple function of the number of preceding blocks to reference and \emph{(ii)} an age preference parameter. A preference for old or high degree vertices creates a graph architecture that is locally tree like and features vertices of arbitrary high degree, which differs extremely from the linear structure of Blockchain-based DLT. Young age preferential attachment interpolates between these two structural extremes and generates a sparse, locally treelike graph in which there are no hubs, i.e.\ the maximal degree grows only weakly with the system. We provide a simple reference rule that allows us to study the two questions above analytically and investigate some transient and asymptotic properties of the transaction block graph such as the reference degree distribution, degree evolutions, the existence of a giant forward component, and the size of the `detached surface' of vertices in the transaction graph that do not connect to any previous transactions.

\subsection{Related literature about DLTs}\label{sec:related_work}

Blockchain technology is based on the following fundamental principles, (i) Decentralization: Due to its distributed architecture, a blockchain does not rely on a central server, but rather stores data in a decentralized manner and can be updated by any node in the network. (ii) Transparency: each node (peer) has a local copy of the blockchain and updates are  sent as broadcast to all nodes. Therefore, any change to the blockchain is transparent to all members of the network. (iii) Open source: the original blockchain protocol introduced with Bitcoin is publicly available and can be used for development. (iv) Autonomy: all peers can update or transmit the stored data. Each node can verify the accuracy of all data received. This eliminates the need for trust between participants. (v) Immutability: once a transaction block is part of the blockchain, it is stored publicly forever and its contents (i.e., transactions) cannot be edited without changing the hash value stored in all subsequent blocks. (vi) Anonymity: Data transmission in the blockchain is done via addresses, while user identities remain hidden.

Current blockchain-based cryptocurrencies are often presented as an alternative to traditional currencies. However, there are questions about the scalability and performance of distributed ledgers compared to established payment systems~\cite{croman2016scaling,pappalardo2017blockchain}. Currently, the Bitcoin network is only capable of processing very few transactions per second \cite{blockchain2018tps} and it takes up to 20 minutes for a newly posted transaction to be processed \cite{blockchain2018act}. Established payment processors, on the other hand, process up to 4 orders more transactions per second and new transactions are usually processed within a few seconds. A significant influence on the performance of the globally distributed peer-to-peer Bitcoin system \cite{donet2014bitcoin} is the propagation speed of the blocks in the network \cite{decker2013information}. For example, for a 1MB block, which is roughly the average block size \cite{blockchain2018abs}, it takes about 2.4 minutes \cite{croman2016scaling} for 90\% of the nodes to receive this block. With the block interval of 15 minutes used by the Bitcoin network, this poses a significant problem, which increasingly leads to the splitting of the blockchain - the so-called blockchain forks \cite{decker2013information}. The emergence of such forks significantly decreases the performance of the distributed ledger system.

Although approaches already exist to optimize the propagation speed of information in the distributed ledger network \cite{fadhil2016bitcoin,fadhil2017locality}, these also do not reach the performance of current financial transaction systems. In addition to optimizing the propagation speed, other improvements arise in the communication protocol used \cite{gobel2017increased, croman2016scaling}. One way to optimize the current Bitcoin protocol in terms of transmission speed is to minimize unnecessary transmissions by adjusting the sequence of exchanged messages needed to transmit new blocks. Instead of first publishing the availability of information, blocks that do not exceed a certain threshold in terms of size can thus be distributed directly to neighboring nodes. This can increase the effective transmission speed, especially for small blocks where the protocol overhead has a particularly large impact. The authors of \cite{decker2013information} have shown that such protocol adaptations can reduce blockchain forks by about 50\%.

To increase scalability and transaction throughput, nonlinear graph-based distributed ledgers \cite{yeow2017decentralized,sompolinsky2016spectre,popov2016tangle,lewenberg2015inclusive,sompolinsky2020phantom} replace the traditional blockchain with a global directed acyclic graph (DAG) that contains all recorded transactions. In this graph, the nodes are transaction blocks (consisting of one or more transactions), while the edges indicate the validation of previous blocks / transactions. Each added node should validate $k \ge 1$ transactions where $k$ can be fixed of random. A key advantage of DAG-based systems is that by deviating from the linear longest blockchain rule, scalability can be greatly improved. The longest blockchain rule essentially requires all nodes to know about newly created blocks quickly, thus limiting the creation of blocks to allow for the propagation of this knowledge before a new block is created. Instead of trusting the longest chain rule as given in blockchain distributed ledgers, DAG-based distributed ledgers rely on metrics such as the largest strongly connected cluster of transaction nodes in the transaction graph to determine the trusted transactions~\cite{sompolinsky2020phantom}.

Beyond financial transactions there exit many application areas for distributed ledgers, for example, in the context of distributed and secure storage of personal data~\cite{zyskind2015decentralizing}  or for asynchronous messaging services in the context of distributed computing~\cite{yin2018hyperconnected} . Especially the former use case is of high relevance due to increasingly frequent incidents in which personal data is publicly accessible due to security breaches or data misuse. Here, the use of distributed ledgers as a system for distributed access control offers a potential solution. Instead of storing data centrally and granting direct access, queries are mapped as well as sharing information via transactions. Combined with off-chain storage of sensitive data, it is thus possible to develop a distributed ledgers system in which each user retains full control over their own data. In addition to increased security, this has the advantage that laws and other regulations can be implemented directly in the protocol used in the form of policies~\cite{zyskind2015decentralizing}.

\subsection*{Organisation of the paper} In the following section, we introduce the model formally and present our mathematical main results. Section~3 contains further discussion of how our results relate to previous work and of the practical implication of our results for DLT applications. Section~\ref{sec:simulation} contains numerical simulations. Finally, detailed proofs of our main results are provided in Section~6.

\section{Model and main results}\label{sec:results}
\vspace{10pt}

\subsection{Model definition}
The model we propose is based on \emph{young-age preferential attachment (YAPA)}. We study a recursive growth parametrized by two numbers, the \emph{edge density} parameter $\alpha>0$ and the \emph{reinforcement bias} $\beta>0$. We use the shorthand notations $[n]$ for $\{1,\dots,n\}$ and $x\wedge y$, $x\vee y$ to denote $\min(x,y)$ and $\max(x,y)$, respectively. At step $n=1$, the graph $G_1$ is initialized with a vertex labelled $1$. At any step $n>1$, one vertex labelled $n$ is added to the system and $G_{n}$ is obtained recursively from $G_{n-1}$ according to the following attachment rule: for each $m\in[n-1]$, $n$ sends an arc to $m$ independently of everything else with probability
\begin{equation}\label{eq:attachment_probability}
	\frac{\alpha}{n-1} \left(\frac{m}{n-1}\right)^\beta  \wedge 1.
\end{equation}

We formally write $G_n=([n], E_n)$ with $E_n\subset [n]\times[n]$ for the graph obtained upon adding the $n$-th vertex. If $n$ sends an arc to $m\in[n-1]$, we denote this event by $n\to m$. The notation $n \twoheadrightarrow m$ is shorthand for the existence of a directed path from $n$ to $m\in[n-1]$. Note that the direction of the arcs along any path can be recovered from the labels of the vertices that belong to the path.

\begin{figure}
	\centering
	\begin{subfigure}[b]{0.32\textwidth}
		\centering
		\includegraphics[width=\textwidth]{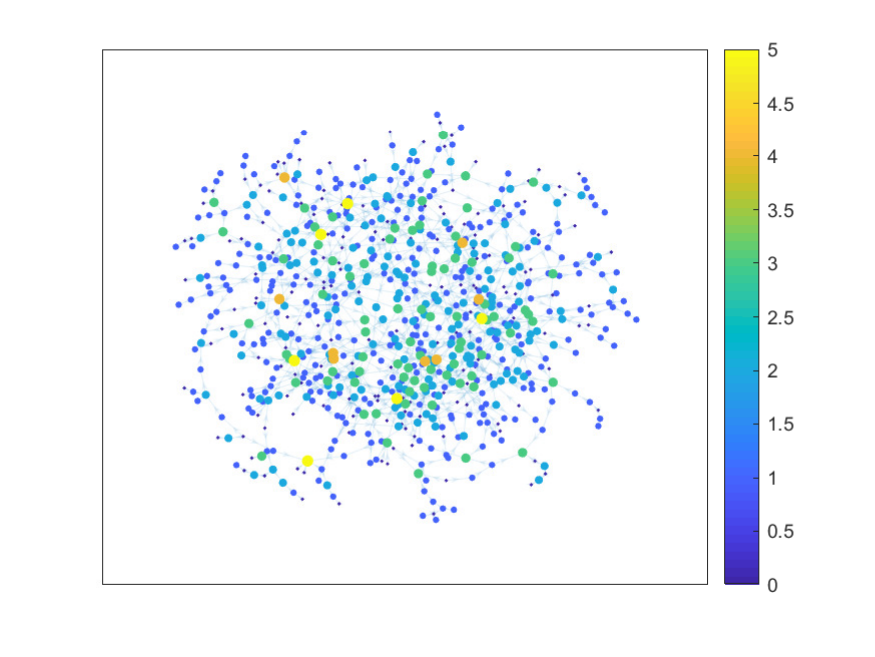}
		\caption{$\alpha=3$}
		\label{fig:graph-alpha3}
	\end{subfigure}\hfill
\begin{subfigure}[b]{0.32\textwidth}
		\centering
		\includegraphics[width=\textwidth]{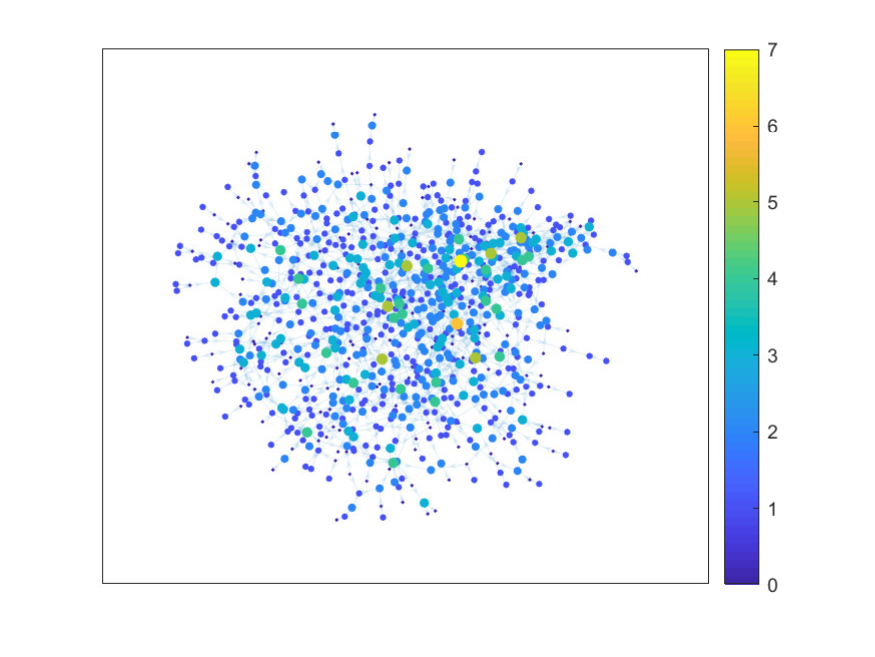}
		\caption{$\alpha=3.5$}
		\label{fig:graph-alpha3.5}
	\end{subfigure}\hfill
\begin{subfigure}[b]{0.32\textwidth}
		\centering
		\includegraphics[width=\textwidth]{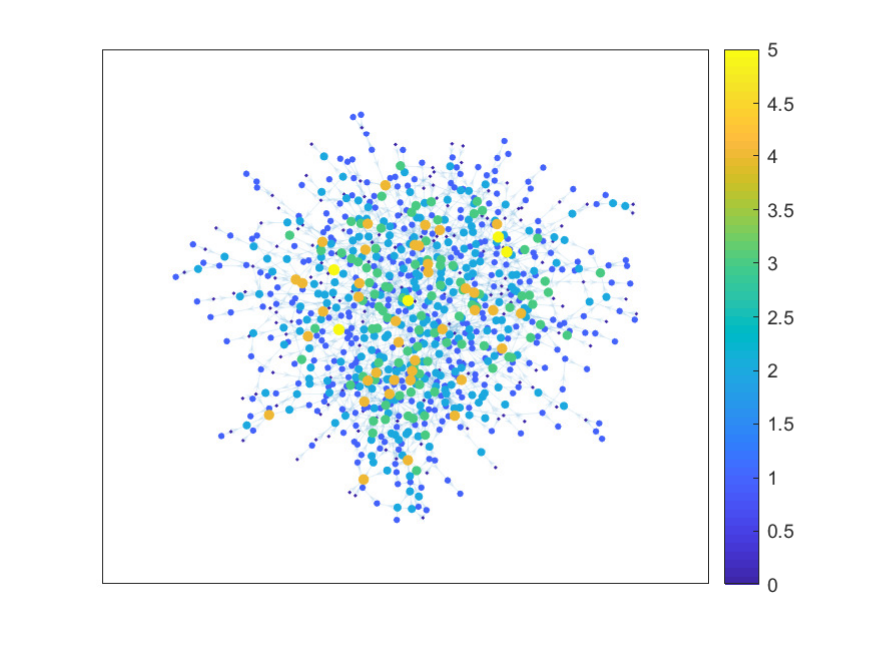}
		\caption{$\alpha=4$}
		\label{fig:graph-alpha4}
	\end{subfigure}
	\captionsetup{width=1.0\linewidth}
	\caption{A depiction of sample realizations of the young age preferential attachment model with reinforcement bias $\beta=2$ and edge density $\alpha\in\left\{3,3.5,4\right\}$. The node size and color correspond to the node indegree.}
	\label{fig:graph-alpha-var}
\end{figure}

\subsection{Key properties of the YAPA model}
Ignoring the dynamical aspect of our model, it can be viewed as an instance of the directed inhomogeneous random graph model proposed in \cite{CaoOlver20}, which allows us to infer some basic properties of the network from known results about this model class. We first establish that the young-age preferential attachment model produces a sparse graph, i.e.\ the number of arcs is of the same order as the number of vertices.
\begin{proposition}[Sparsity]\label{prop:sparsity} It holds that
	$$\lim_{N\to \infty}\frac{|E_N|}{N} =   \frac{\alpha}{\beta+1}\quad \text{ in }L^1\text{ and almost surely}.$$	
\end{proposition}

Next, we determine the distributional limit of the vertex degrees in $G_n$. For any vertex $m\in[n]$, we denote by $D_n(m)=(D_n^{\mathsf{in}}(m),D_n^{\mathsf{out}}(m))$ its degree (vector) in $G_n$, i.e.\ \[D_n^{\mathsf{in}}(m)=|\{k: k\to m \text{ in } G_n \}|\text{ and } D_n^{\mathsf{out}}(m)=|\{k: k\leftarrow m \text{ in } G_n \}|.\]
Note that the outdegree of a vertex is determined once and for all upon its insertion, i.e.\ $D_n^{\mathsf{out}}(m)=D_m^{\mathsf{out}}(m)=:D^{\mathsf{out}}(m)$ for all $n\geq m$. As can be guessed from the fact that connections occur independently and with probability of order $1/n$, the asymptotic degree distribution is bivariate Poisson.
\begin{theorem}[Limiting degree distribution and rate of convergence]\label{thm:typicaldegree}
	Let $G_N$ denote the young-age preferential attachment graph on $N$ vertices. For $U\in[N]$ uniformly chosen among the vertices of $G_N$ the following asymptotic results holds $$D_N(U)=(D_N^{\mathsf{in}}(U),D^{\mathsf{out}}(U))\longrightarrow (Z^{\mathsf{in}},Z^{\mathsf{out}}) \quad \text{ in distribution and in }L^1 \text{ as }N\to\infty.$$
	Here, $Z^{\mathsf{in}}$ and $Z^{\mathsf{out}}$ are independent and $Z^{\mathsf{out}}$ is $\operatorname{Poisson}(\alpha/(\beta+1))$-distributed. The limiting indegree $Z^{\mathsf{in}}$ has distribution $\operatorname{mixed-Poisson}(\Lambda)$ with mixing random variable $\Lambda=\alpha(1-V^\beta)/\beta$, where $V$ is $\operatorname{Uniform}[0,1]$. Moreover we have that
	\begin{equation}\label{eq:totalvariation}
		\mathsf{d}_{\textup{TV}}(P^{\mathsf{in/out}}(N),p^{\mathsf{in/out}})= O\left(\sqrt{\frac{(\log N)^3}{N}}\right), \quad \text{almost surely as }N\to\infty,
	\end{equation}
	where $P^{\mathsf{in/out}}(N)=(P^{\mathsf{in/out}}_{k}(N))_{k=0}^\infty$ denotes the empirical in- or outdegree distribution of $G_N$ and $p^{\mathsf{in/out}}=(p^{\mathsf{in/out}}_k)_{k=0}^\infty$ denotes the distribution of $Z^{\mathsf{in/out}}$ and $\mathsf{d}_{\textup{TV}}(\cdot,\cdot)$ denotes total variation distance.
\end{theorem}
\begin{remark}
	The form of the limiting degree distribution in Theorem~\ref{thm:typicaldegree} is obtained rather directly by applying the machinery of \cite{CaoOlver20}, but the error bound \eqref{eq:totalvariation} is new and does not follow from previous results about inhomogeneous random graphs. Inspection of the proof further yields that an error of the same order can be achieved for any graph on $N$ vertices in which edges are drawn independently with probability at most $c/N$ for some universal $c$, cf.\ Lemma~\ref{lem:concentrationbound}. We do not believe that the logarithmic term is asymptotically sharp, but the polynomial order $1/\sqrt{N}$ is consistent with known results for other sparse models with independent edges such as the Erd\H{o}s-Rényi-Graph, cf.\ \cite[Theorem 5.12]{hofstad_book}.
\end{remark}

It follows from Theorem~\ref{thm:typicaldegree} that whereas the outdegree is asymptotically independent of the time at which a node enters the system, the indegree is not. Essentially, the limiting indegree distribution of a given node is also $\operatorname{Poisson}(\alpha/(\beta+1))$ but the time needed to collect all incoming links grows proportionally with network size. A node born at the beginning corresponds to $V$ near $0$ and a node born at a time close to the total age of the system corresponds to $V$ close to $1$.

\medskip

The indegree dynamics of fixed vertices are further characterized by the following bounds:
\begin{proposition}[Degree evolutions]\label{thm:degev}
	The maximal (total) degree of any vertex in $G_N$ is almost surely $O(\log N)$. We further have that
	\begin{enumerate}[(i)]
		\item For any $n\geq m\geq 1$, $$\PP[D_n^{\mathsf{in}}(m) > x]\leq \exp\left[-\frac{\alpha}{\beta}-x\left(\log x +\log \frac{\beta}{\alpha}-\frac{\beta}{\alpha} \right)\right],\; x>\alpha/\beta;$$
		\item $\EE[D^{\mathsf{out}}(m)]=\frac{\alpha}{\beta+1}+O(1/m)$ for any $m\geq 1.$
	\end{enumerate}	
\end{proposition}

Let us now discuss the connectivity of the YAPA network. The \emph{forward component} of $m\in [n-1]$ after the arrival of vertex $n$, $1\leq m \leq n $ is given by
\[
C_n(m)=\{k: k\twoheadrightarrow m \text{ in } G_{n} \}\cup\{ m \},
\]
For the distributed ledger application discussed in Section~\ref{sec:intro}, it is desirable that the network be {rooted}, i.e.\ $C_n(1)=[n]$. Due to the stochastic nature of the attachment mechanism this is typically not the case but can formally be achieved by augmenting the network by some root vertex $\rho$ that deterministically connects to precisely those vertices which have outdegree $0$. Of course, the proportion of vertices which fail to connect to any existing vertices should be rather small, which is ensured by the light tails of the limiting Poisson distribution.

\begin{corollary}[Detached surface]\label{cor:detsurf}
	We have that
	$$\lim_{N\to\infty}\frac{|\{m\in\{2,\dots,N\}:D^{\mathsf{out}}(m)=0\}|}{N} = \textup{e}^{-\alpha/(\beta+1)} \text{ almost surely and in } L^1.$$
\end{corollary}
While there is still a positive fraction of nodes with no left neighbors, this fraction is not very large: for instance, if we would like to make sure, that with probability at least $.995$ a typical vertex sends an arc to at least one vertex, it suffices to choose $\alpha>\log(200)(\beta+1)\approx 5.3 (\beta+1)$.

Next, we would like to ensure that the vast majority of vertices which do connect upon arrival in the system is contained in the forward component of only a handful of vertices. This can only be achieved, if there exist vertices $m$ such that $C_N(m)$ grows linearly with $N$. Our first result concerning such macroscopic components states that the forward component of a typical vertex is not of linear order.
\begin{proposition}[No giant forward components for typical vertices]\label{prop:nogiant}
	Let $U\in[N]$ be uniformly chosen, then for any $\varepsilon>0$, we have
	\[
	\lim_{N\to\infty}\PP(|C_N(U)|\geq \varepsilon N)=0.
	\]
\end{proposition}
\begin{remark}\label{rmk:locallimit}
	As for Theorem~\ref{thm:typicaldegree}, the neighborhood around a typically chosen vertex can be described by the local limit approach used in \cite{CaoOlver20}. It is not hard to see, that the weak local limit of $C_N(U)$ is in fact the family tree obtained from the multitype Galton-Watson-Process given by the following recursion:
	\begin{itemize}
		\item Sample $V\in[0,1]$ uniformly. Generation $0$ of the tree comprises the root, which has type $V$.
		\item For any $k\geq 1$, given generation $k-1$ of the tree, each member $m$ in generation $k-1$ of type $V_m$ independently produces a \textsf{Poisson}($\alpha(1-V_m^\beta)/\beta$)- distributed number of offspring to obtain the $k$-th generation and conditional on $V_m=x\in(0,1)$, the offspring types are independent with density
        \[
        q_x(y)=\frac{\beta x^\beta}{1-x^\beta}\,y^{-\beta-1}
        \mathsf{1}\{x<y<1\}.
        \]
        Equivalently, the offspring types form a Poisson point process on $(x,1)$ with intensity
        $\alpha x^\beta y^{-\beta-1}\,\textup{d}y$. A particle of type $1$ has no offspring.
	\end{itemize}
\end{remark}
The proof of Proposition \ref{prop:nogiant} in Section~\ref{sec:proofs} shows that the local limit tree dies out for any choice of $\alpha,\beta>0$. Contrary to the local behavior near a typical vertex, if $\alpha>\beta+1$, most vertices indeed belong to the forward components of a few old vertices. More precisely, let $n_0=\lfloor 1/\alpha \rfloor$, then the vertices $1,\dots,n_0$ are always sequentially connected by construction and we thus focus on the forward components of vertices added after $n_0$.
\begin{theorem}[Existence of giant forward components]\label{thm:giant}
	Let $m>n_0$ be given. Then almost surely
	\[
	\lim_{N\to\infty}\frac{|C_N(m)|}{N}=\begin{cases}
		0, &\text{if }\alpha\leq \beta+1,\\
		Z(m), & \text{if }\alpha> \beta+1,
	\end{cases}
	\]
	where $Z(m)$ is a non-trivial random variable that is concentrated on the two solutions $0$ and $\gamma>0$ of the fixed point equation
	\[
	x = 1-\textup{e}^{-\frac{\alpha}{\beta+1}x}, \quad x\in[0,\infty).
	\]
\end{theorem}
Theorem~\ref{thm:giant} is the main result of this article. The apparent difference in the characterisations obtained in Proposition~\ref{prop:nogiant} and Theorem~\ref{thm:giant}, respectively, are due to the reducibility of the graph. Proposition~\ref{prop:nogiant} is concerned with the forward component of a \emph{typical}, i.e.\ uniformly chosen, vertex. Such a typical vertex has a birth time of order $N$ and can simply not collect enough direct and indirect references up to time $N$ to develop a large forward component. In particular, applying the machinery of \cite{CaoOlver20} (after some adaptations to account for irreducibility), yields that typical vertices do not have a giant forward component, see the proof of Proposition~\ref{prop:nogiant}.

In contrast, Theorem~\ref{thm:giant} states that if an \emph{early} vertex is kept fixed and the time horizon is extended to infinity, with positive probability, the selected vertex will be the root of a forward component of linear size eventually. We stress, that it is not possible, to derive statements of this type directly using the general theory of \cite{CaoOlver20}, which uses the local weak limit of the graph (i.e.\ limit distributions of neighbourhoods of \emph{typical} vertices) and therefore does not provide sufficient information about vertices with index $n=o(N)$.

\section{Further related work and discussion of results}
\vspace{10pt}

\subsection{Design of Tip Selection for DAG-type Distributed Ledgers }\label{sec:tangle_implications}
DAG-type Distributed Ledgers such as IOTA~\cite{popov2016tangle} utilize different algorithms for achieving consensus. State-of-the-art DAG-type DLT repurpose the act of selecting transactions for validation, i.e., constructing arcs from an arriving node to previous ones, for achieving consensus. In a nutshell, the tip selection algorithm in \cite{popov2016tangle} is based on a biased random walk and that a transaction which is directly or indirectly validated by a large portion of the subsequent transactions is highly probable to be considered valid.
Precisely, a tip on the surface of the DAG is selected based on a random walk from genesis to the surface. This random walk is biased in the sense that nodes possess a weight that resembles the number of subsequent nodes that are directly or indirectly attached to this node and that the jump probability of the random walk from a node $i$ with weight $w_i$ to a node $j$ with weight $w_j$ is $e^{-\eta(w_i-w_j)}$  with a given constant $\eta>0$.
For finite time the  value of the constant $\alpha$ may play a role to adjust the tip selection to mimic a uniformly random tip selection or to quickly concentrate the tip selection on a set of young nodes.
However, this algorithm is different from our YAPA model as it concentrates the tip selection on certain walks where the node weights are large.

Further works such as \cite{Kusmierz2018,Ferraro2020} analyze different properties of the DAG-type Distributed Ledgers especially the IOTA DAG, that is denoted the Tangle.
In \cite{Kusmierz2018} the authors derive approximations of the probabilities that a node is left behind or that a node becomes a permanent tip. Specifically, the first approximation of the probability that a node is left behind is similar to the detached surface evolution in Corollary \ref{cor:detsurf}.
The authors of \cite{Ferraro2020} provide a modified IOTA tip selection algorithm that is based on a combination of two Markov Chain Monte Carlo algorithms with different parameters that aims that all nodes are validated. They prove this property using a fluid limit in the arrival rate of transactions.

In Coordicide \cite{Coordicide2020}, which is a follow up work on IOTA constituting a modern DAG-type DLT, tip selection is fixed to ``\emph{validating two existing transactions}''.
Here, the authors dispense with tip selection as a consensus mechanism and rather consider it from a performance point of view, i.e., analyzing the impact of the tip selection on the DAG properties, especially, in discouraging ``lazy'' behavior which resembles the recurring validation of a fixed set of old transactions.

We consider this work in the same vein, i.e., our YAPA model provides a relationship between the parameters of the tip selection algorithm, i.e., \emph{the edge density  $\alpha$ and the reinforcement bias $\beta$}, and the DAG distributed ledger properties.
We establish the YAPA model in \eqref{eq:attachment_probability} as a rule to encourage the validation of young transactions.
From the calculation of the asymptotic distribution of the outdegree $Z^{\mathsf{out}}$ (which turns out to be $\operatorname{Poisson}(\alpha/(\beta+1))$) as well as the non-asymptotic expected outdegree at a given time point $\mathsf{E}\left[D_k^{\mathsf{out}}\right]$ we know the expected number of transactions that are validated given parameters $\alpha,\beta$.
Further, the  reinforcement bias $\beta>0$ is understood as a forgetting factor, precisely, a parameter to set the preference for validating young transactions, i.e., larger  $\beta>0$ leads to a stronger preference of younger transactions over older ones.
Hence, by choosing $\alpha,\beta$ while ensuring $\alpha > \beta +1$ the properties of the distributed ledger can be controlled. As is reflected in our results, we study the following properties: \emph{(i)} the emergence of a `giant' components, i.e., a linear increase of a the number of direct and indirect validations that a transaction receives over time, \emph{(ii)} the average number of direct validations received by a transaction over time, and \emph{(iii)} the number of orphan transactions.

\subsection{Related mathematical results}
Our model is similar in spirit to the one introduced in \cite{LyonMahmo20}, but with two important differences. Firstly, their model produces a tree, whereas our model produces a directed acyclic graph that is typically not a tree. Secondly, we modulate the attachment via the parameters $\alpha$ and $\beta$ to obtain a whole range of instances. This flexibility allows to tune the model to specific applications and also allows to observe the connectivity phase transition of Theorem~\ref{thm:giant}. The young-age preferential attachment tree of \cite{LyonMahmo20} corresponds to running our model with $\beta=1$ and conditioning on the creation of precisely one arc per time step. As explained in Section 2, our model can be analysed in the general framework of \cite{CaoOlver20} who adapted the very general framework of inhomogeneous random graphs developed in \cite{BolloJansoRiord07} to directed graphs, thereby generalising earlier results of \cite{BloznGoetzJawor12}. We use some techniques from the (directed) inhomogeneous random graph world to obtain the degree sequence and analyse the emergence of a giant connected component for our model in Proposition~\ref{prop:nogiant}. However, since $\kappa(\cdot,\cdot)$ is reducible and we investigate connected components in terms of directed paths our main result Theorem~\ref{thm:giant} cannot be obtained by a direct application of the results in \cite{BloznGoetzJawor12,CaoOlver20}. We also emphasize the dynamic nature of our model which is not captured by the inhomogeneous random graph framework. Furthermore, if we were to ignore the direction of the arcs and consider $G_N$ as an undirected graph, then it is straightforward to deduce that the limit $\lim_{N\to\infty}|C_N(U)|/N$ in Proposition~\ref{prop:nogiant} would be positive whenever $\alpha>\beta+1$, c.f.\ \cite{BolloJansoRiord07}.

\medskip

The limiting case $\beta=0$ of the YAPA model corresponds to a directed version of the uniformly grown random graph investigated in \cite{BolloJansoRiord04} and \cite{BolloRiord04}, this is a very natural recursive model that was first studied by Dubins in the 1980's, cf.\ \cite{KalikWeiss88,Shepp89,DurreKeste90}. Although we exclusively consider positive values of $\beta$, choosing $\beta<0$ would lead to a preference for old-age vertices and produce an effect similar to degree-based preferential attachment models such as the Barab\'asi-Albert model \cite{AlberBarab99}. Models of this type have been investigated thoroughly in the last twenty years, see e.g.\ \cite{BolloRiordSpencTusna01,DereiMoert09,DereiMoert13} and the references therein. Old-age preferential attachment models produce graphs with entirely different features such as a scale-free degree distribution, whereas the degrees in the YAPA model are fairly sharply concentrated around their expectation as is exemplified by Proposition~\ref{thm:degev}.

\section{Simulation}\label{sec:simulation}

\vspace{10pt}

In the following, we  describe evaluation results that are based on simulating the model with attachment probabilities given by~\eqref{eq:attachment_probability}. Transaction blocks arrive according to a Poisson process with rate $\lambda=1$. We fix the reinforcement bias $\beta=2$ and variate the edge density parameter $\alpha$. Fig.~\ref{fig:component_size_alpha2} - Fig.~\ref{fig:component_size_alpha_var} show the evolution of the connected component size starting at the genesis node (transaction) for the first $10^5$ transactions for different $\alpha$. Note the linear growth rate of the connected component for $\alpha>\beta+1$.
The figures show the empirical average in addition to the standard deviation (which is depicted as the shaded area) from $20$ independent simulation runs. Note that the standard deviation for the subcritical regime $\alpha=2$ is large because there is no natural scaling (in $N$) for the component size.

\begin{figure}
	\centering
	\begin{subfigure}[b]{0.49\textwidth}
		\centering
		\includegraphics[width=\textwidth]{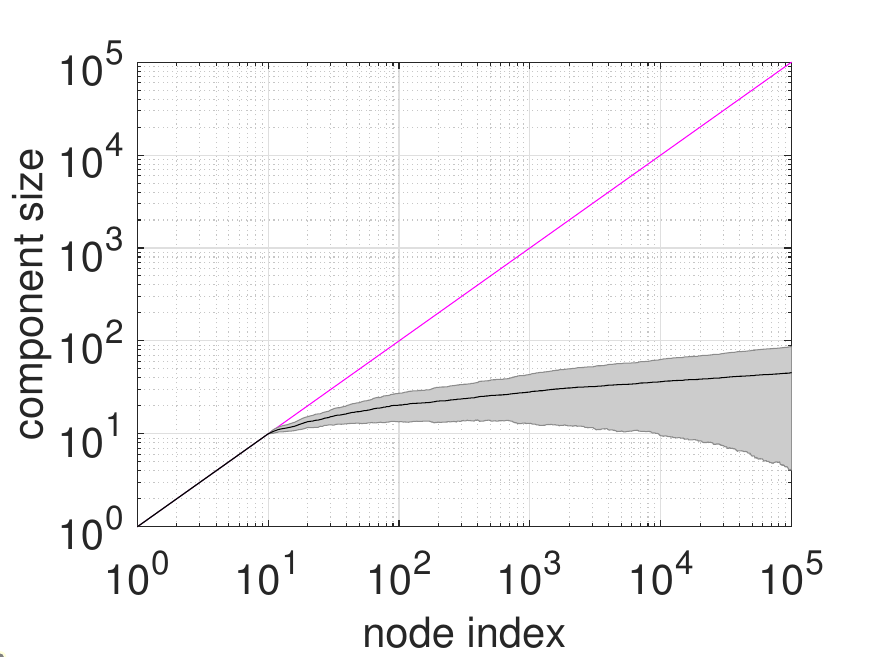}
		\caption{$\alpha=2$}
		\label{fig:component_size_alpha2}
	\end{subfigure}\hfill
\begin{subfigure}[b]{0.49\textwidth}
		\centering
		\includegraphics[width=\textwidth]{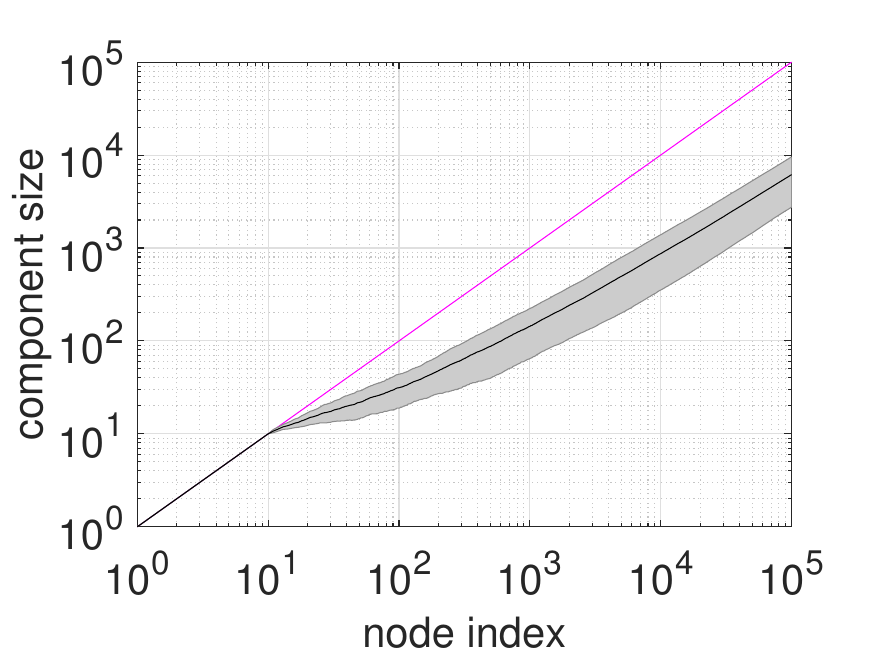}
		\caption{$\alpha=3$}
		\label{fig:component_size_alpha3}
	\end{subfigure}\par\medskip
\begin{subfigure}[b]{0.49\textwidth}
		\centering
		\includegraphics[width=\textwidth]{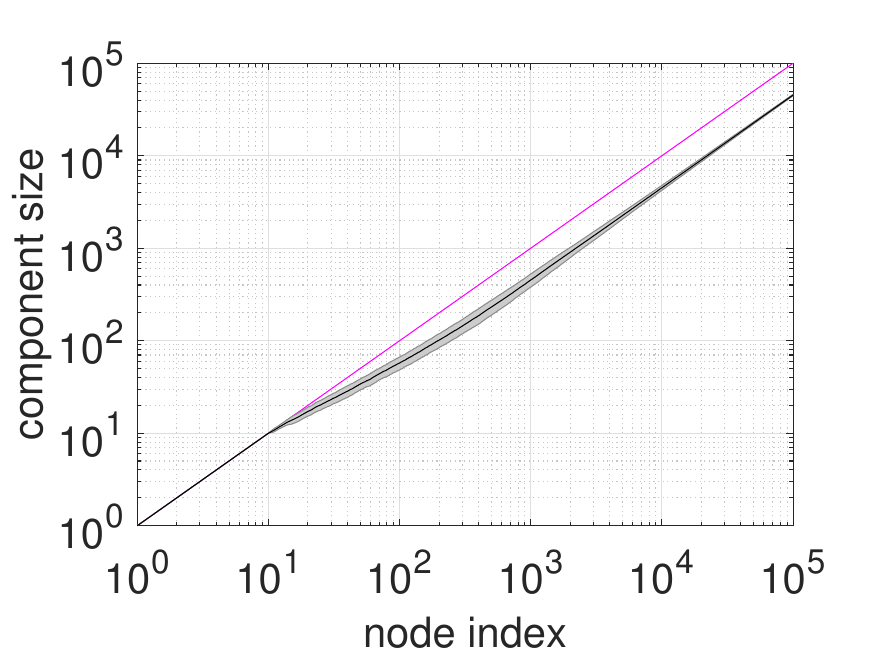}
		\caption{$\alpha=4$}
		\label{fig:component_size_alpha4}
	\end{subfigure}\hfill
\begin{subfigure}[b]{0.49\textwidth}
		\centering
		\includegraphics[width=\textwidth]{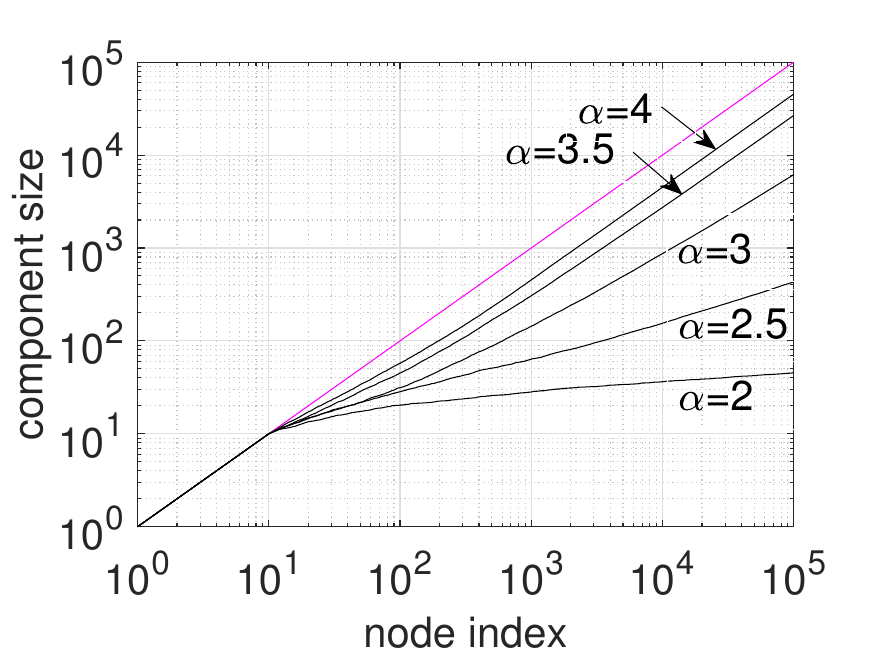}
		\caption{variable $\alpha$}
		
		\label{fig:component_size_alpha_var}
	\end{subfigure}
	\captionsetup{width=1.0\linewidth}
	\caption{Evolution of the connected component size (starting at the genesis node) for reinforcement bias $\beta=2$ and variable edge density parameter $\alpha\in\left\{2,3,4\right\}$. As an initial condition we assume the first $10$ nodes reference the genesis node.}
	\label{fig:component-size_beta2-alpha-var}
\end{figure}

\begin{figure}
  \centering
  \includegraphics[width=0.5\textwidth]{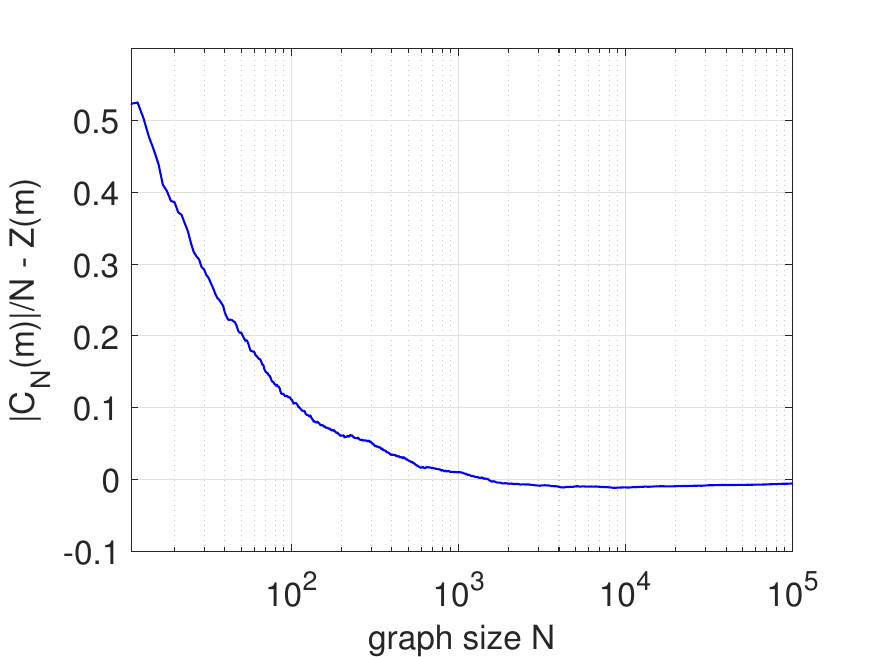}
		\caption{Numerical verification of Thm.~\ref{thm:giant} for $\alpha=4, \beta=2, m=10$. The error $|C_N(m)|/N  - Z(m)$ diminishes with increasing number of vertices $N$.}
\label{fig:thm2_error}
\end{figure}

Figure~\ref{fig:thm2_error} shows the error $|C_N(m)|/N  - Z(m)$ as calculated from Thm.~\ref{thm:giant} for $\alpha=4, \beta=2, m=10$. The figure shows clearly the convergence of the ratio $|C_N(m)|/N$ to $Z(m)$ with a growing number of vertices $N$.

\begin{figure}
	\centering
	\begin{subfigure}[b]{0.48\textwidth}
		\centering
		\includegraphics[width=\textwidth]{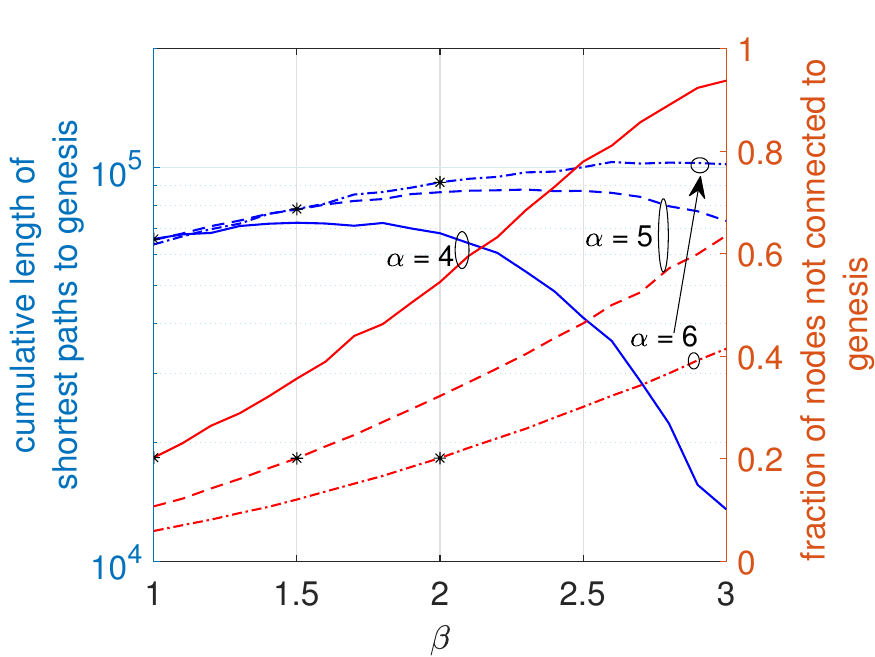}
		\caption{Cumulative length of shortest \\paths to genesis in hops. }
		\label{fig:sum_shortstpaths}
	\end{subfigure}\hfill
\begin{subfigure}[b]{0.48\textwidth}
		\centering
		\includegraphics[width=\textwidth]{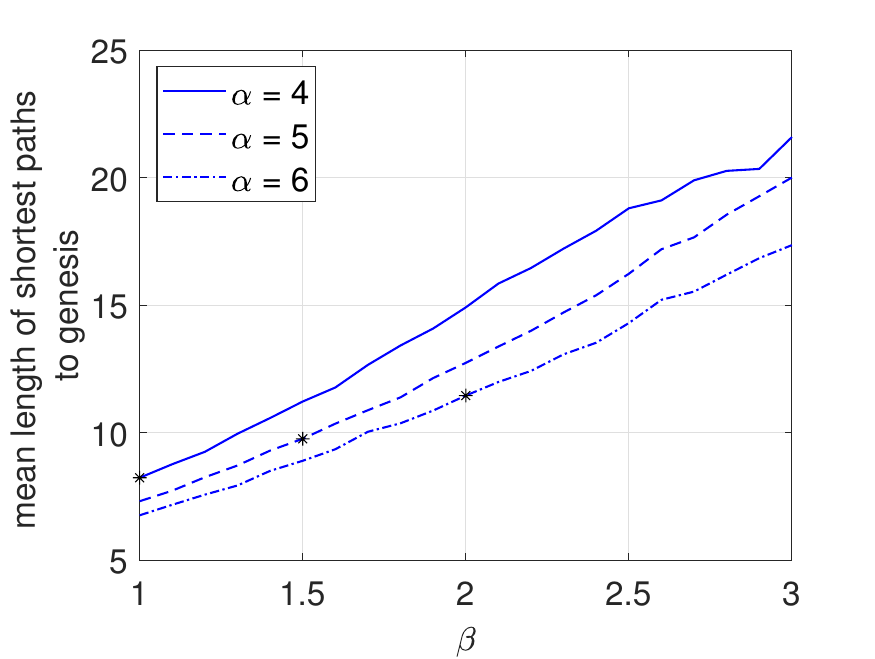}
		\caption{Mean length of shortest paths \\ to genesis in hops.}
		\label{fig:mean_shortstpaths}
	\end{subfigure}
	\captionsetup{width=1.0\linewidth}
	\caption{Properties of the shortest paths to genesis for $\alpha \in \{4,5,6\}, \beta \in \{1,\dots,3\}$. The width and depth of the constructed graphs show a strong dependency on $\beta$. The figures are obtained from $10^2$ simulation runs each including $n=10^4$ transactions. The black markers correspond to a mean outdegree $\mathsf{E}\left[Z^{\mathsf{out}}\right]=2$. }
	\label{fig:properties_shortest_paths}
\end{figure}

Figure~\ref{fig:properties_shortest_paths} shows the properties of the shortest paths to the genesis node for various values of the reinforcement bias $\beta$ and the edge density parameter $\alpha$.
In light of the impact of the tip selection algorithm parameterization through $\alpha$ and $\beta$ the figure reflects the \emph{shape} of the transaction DAG, i.e.,  its depth and width in terms of the length of the shortest paths to genesis.
In comparison, note that the extreme cases of a DAG of depth $1$ or a totally ordered chain produce sum lengths of the shortest paths of size $n$ and $\frac{n^2}{2}$, respectively.
The figure also shows markers for the values of $(\alpha,\beta)$ that correspond to a mean outdegree $\mathsf{E}\left[Z^{\mathsf{out}}\right]=2$ which is used in the context of DAG-type DLT \cite{popov2016tangle,Coordicide2020} as a \emph{deterministic} guideline for the number of validated nodes per new node.
While the fraction of nodes not connected to genesis remains constant as given by the detached surface approximation in \eqref{eq:detached_surface_probability_approx} the shape of the DAG DLT can be controlled by the choice of $\alpha,\beta$.
Note that a similar argument can be made in terms of the shortest paths to some arbitrary Genesis cluster of nodes.

\section{Implications for DLT applications}

Our work is motivated by questions on the design choices of DAG-based DLT applications, in particular, on how the analyzed YAPA referencing scheme that relates newly generated blocks to previously generated ones affects the DLT properties.
In the following, we discuss the implications of the obtained results for some design aspects of DLT applications.

\textit{Tip selection:}
Tip selection is often convolved with consensus as in \cite{popov2016tangle,mueller:tangle2}. In \cite{Coordicide2020} tip selection is considered separately and set to referencing a fixed number of previous blocks.
When Tip selection is convolved with consensus it is usually based on some aggregate weight of the blocks that must be tracked~\cite{popov2016tangle,mueller:tangle2}.
Some works also propose uniform random tip selection on a subset of (non-conflicting) blocks such as~\cite{mueller:tangle2}.

The first implication of the YAPA model definition is that users that generate blocks do not have to track or calculate block weights according to more or less complex and time consuming algorithms.
The configuration model is entirely determined by the YAPA parameters $\alpha,\beta$.
In turn, these determine the required storage and computation capacity for reference operations, i.e., the number of references (hashes) that are required on average as in \Cref{prop:sparsity}.
Note that each reference is usually tied to one hash operation as well as a linear growth of storage space.
This stands on contrast to the deterministic number of references as well as the weight-based referencing methods from above.

\textit{Laziness and permanent tips:}
The reviewed state of the art makes great effort to avoid lazy behavior, i.e., preventing users from referencing a fixed set of old blocks.
The temptation to do this is based on the reuse of computed hashes in old blocks in referencing for new blocks.
The YAPA configuration readily dispenses with this problem for $\alpha,\beta>0$.
On the other hand, the risk of a block becoming a permanent tip is shown for our model in \Cref{cor:detsurf}.
This can be considered the cost of using the simple YAPA configuration model as the detached surface contains transaction blocks that are never verified.

\textit{Consensus and Time to Confirmation:}
The results of this work emphasize that consensus in a DAG-based DLT that uses the proposed YAPA configuration model can be reached. \Cref{thm:giant} shows that a consensus method based on the so-called heaviest weighted subchain~\cite{Sok:2020} or the so-called heaviest sub-DAG~\cite{mueller:tangle2} is viable under a proper YAPA configuration of $\alpha > \beta+1$.
The simulation results from above validate this behavior for a range of realistic parameters, i.e., average number of references/ hashes of $\alpha\in\{2,3,4\}$ and $\beta=2$.
Further, the simulation results provide insights in the form if the DLT that is controlled by the YAPA configuration $\alpha,\beta$.
Related works such as~\cite{mueller:tangle2,popov2016tangle} assume a constant Tangle width and that the number of tips of the Tangle are stationary with a given mean.
Note that in~\cite{mueller:tangle2} this is justified by a number of chosen tips equal to 1.
Our simulations show on the other hand that the shape of the Tangle can be controlled by adjusting $\alpha$ and $\beta$ at a fixed average outdegree, i.e., number of referenced previous blocks.
In turn, the shape of the Tangle implicates the causal relations of transaction blocks, i.e., for a given number of blocks $N$ having few long chains of blocks (approaching line topologies) with cross references is semantically different from having many short chains of blocks (approaching a star topology).
The former form distributes the cumulative weight required for consensus while the latter concentrates this weight on a few blocks.
Nevertheless, the result in \Cref{thm:giant} shows that consensus can be achieved after long enough time.

Time to confirmation is a performance metric for the consensus protocol of the DLT.
As discussed above, the YAPA configuration model for a DLT provides two guarantees: First, the confirmation time in terms of the rate of convergence of the distribution of the number of \emph{direct block references received} is bounded in \Cref{thm:typicaldegree}.
This stands in contrast to heuristics to obtain estimates of the time to confirmation as for example in~\cite{mueller:tangle2,popov2016tangle} where some difficulty stems from the non-tractability of the referencing rule.

\section{Proofs}\label{sec:proofs}
The proofs of our main results are divided into three sections. The first section contains the derivation of results that can be obtained in the general framework of \cite{CaoOlver20}. The second section provides concentration bounds for degree properties which are needed in the last section or are interesting in their own right and the final section contains the proof of our main result, Theorem~\ref{thm:giant}.

\subsection{Directed inhomogeneous random graphs and branching process approximation}
We call
\[
w_{n-1}(m)=\alpha\left(\frac{m}{n-1}\right)^\beta, \quad 1\leq m+1 \leq n,
\]
the \emph{weight} of vertex $m$ upon the arrival of vertex $n$. An equivalent way of characterising our model is to fix $N\in\mathbb{N}$, and rewrite the connection probabilities using the \emph{kernel} $\kappa:[0,1]^2\to[0,\infty)$ given by
\begin{equation}\label{def:kappadef}
	\kappa(x,y)=\alpha x^\beta y^{-(\beta+1)} \mathsf{1}\{x<y\}.
\end{equation}
Then \[
\PP(n\to m \text{ in } G_N)=\left(\frac{\kappa(m/N,n/N)}{N}\left(1+\frac{1}{n-1}\right)^{\beta+1} \right)\wedge 1, \quad 1\leq m < n < N.
\]
Setting also $$\phi\left(\frac mN,\frac nN\right)=\left(1+\frac{1}{n-1}\right)^{\beta+1}-1\in O(1/n),$$
the pair $(\kappa,\phi)$ parametrizes an instance of a \emph{directed inhomogeneous random graph} in the sense of Cao and Olvera-Cravioto \cite{CaoOlver20}. We can thus readily apply several results of \cite{CaoOlver20}, by which we immediately establish  sparsity of $G_N$ and the form of the limiting degree distribution in the next subsection. Note that, however, $\kappa$ is clearly a reducible kernel, since it has no mass on the lower diagonal. Therefore the connectivity results for irreducible kernels established in \cite{CaoOlver20} do not apply.

\medskip

It is straightforward to verify the regularity assumptions listed in \cite[Assumption 3.1]{CaoOlver20} for our model, using $\kappa$ and $\phi$ as defined above and interpreting $[0,1]$ together with Lebesgue measure as the ground space. Propositions~\ref{prop:sparsity} is now a special case of the corresponding statement in \cite{CaoOlver20}.

\begin{proof}[Proof of Proposition \ref{prop:sparsity}]
	By \cite[Proposition 3.3]{CaoOlver20}, we have
	\[
	\frac{|E_N|}{N}\to \int_0^1\int_0^1\kappa(x,y)\,\textup{d}x\textup{d}y=\frac{\alpha}{\beta+1}\quad \text{in }L^1.
	\]
	Each \[X_n=|E_{n+1}|-|E_n|=\sum_{m=1}^{n}\mathsf{1}\{n+1\rightarrow m\},\quad n\in\mathbb{N},\]
	is a sum of independent Bernoulli random variables with variance \[\sigma_n^2\leq \sum_{m=1}^{n} \frac{w_{n}(m)}{n}\left[\left(1-\frac{w_{n}(m)}{n}\right)\vee 0\right],\quad n\in \mathbb{N}.\]
	In particular, since $w_{n}(m)\leq \alpha$ for all $m,n$ we have
	\[
	\sum_{n=1}^\infty \frac{\sigma_n^2}{n^2}\leq \sum_{n=1}^\infty \frac{\alpha}{n^2}<\infty
	\]
	and Kolmogorov's Strong Law of Large Numbers (see e.g.\ \cite[Theorem 8.3]{FelerVol2}) yields almost sure convergence.
\end{proof}

\begin{proof}[Proof of limiting distribution in Theorem \ref{thm:typicaldegree}]
	The limiting distribution described in Theorem \ref{thm:typicaldegree} is obtained as a special case of the statement of \cite[Theorem 3.4]{CaoOlver20} when applied to the kernel $\kappa$ given in \eqref{def:kappadef}.
\end{proof}

\begin{proof}[Proof of {{Corollary~\ref{cor:detsurf}}}]
	The \emph{detached surface} $$\mathcal{D}(n)=\{1< k\leq n: D^{\mathsf{out}}(k)=0\}, \quad n>2,$$ comprises all vertices that have arrived and do not attach to any previous vertices. We let $\Delta_n=|\mathcal{D}(n)|$ denote the size of the detached surface at time $n$. $L^1$-convergence of $\Delta_n/n$ to $\textup{e}^{-\frac{\alpha}{\beta+1}}$ is a direct consequence of Theorem~\ref{thm:typicaldegree}, since $\EE[\Delta_n/n]+1/n$ is precisely the probability that a uniformly chosen vertex has outdegree $0$. Alternatively, we can calculate the asymptotic density of $\Delta_n$ directly. The probability that the $k$-th arriving vertex does not attach to any previous vertex is
	\begin{equation}\label{eq:detached_surface_probability}
		p'_k = \prod_{i=1}^{k-1} (1 - p_k(i)),\quad k>1.
	\end{equation}
	Using that $p_k(i)\leq \frac{\alpha}{k-1}$, we obtain that $\log(1-p_k(i))=-p_k(i)+O(1/k^2)$ and thus
	\begin{equation}\label{eq:detached_surface_probability_approx}
		\lim_{k\to\infty}p'_k= \lim_{k\to\infty}\textup{e}^{-\sum_{i=1}^{k-1}p_k(i)}=\textup{e}^{-\frac{\alpha}{\beta+1}}.
	\end{equation}
	The outdegrees of different vertices are independent, and therefore $\Delta_n$ is a sum of independent $\operatorname{Bernoulli}(p'_k)$-distributed random variables. The corresponding variances $p'_k(1-p'_k)$ are uniformly bounded by $1$, thus another invocation of Kolmogorov's Strong Law of Large Numbers yields almost sure convergence of $\Delta_n/n$ to $\textup{e}^{-\alpha/(\beta+1)}.$
\end{proof}

\begin{proof}[Proof of Proposition~\ref{prop:nogiant}]
The exploration coupling in \cite[Section 4.3.1]{CaoOlver20} identifies the local limit of the forward component of a uniformly chosen vertex with the multitype Galton--Watson tree $\mathcal{T}$ described in Remark~\ref{rmk:locallimit}. Its root type $U_0$ is uniform on $(0,1)$, and a particle of type $x$ has offspring types given by a Poisson point process with intensity
\[
\kappa(x,y)\,\textup{d}y
=\alpha x^\beta y^{-\beta-1}\mathsf{1}\{x<y<1\}\,\textup{d}y.
\]
We show that $\mathcal{T}$ is finite almost surely. Fix $\delta\in(0,1)$ and condition on $U_0>\delta$. Types increase along every genealogical path, so all particles then have types in $[\delta,1]$. On this interval the offspring kernel is bounded by
\[
0\leq\kappa(x,y)\leq M_\delta:=\alpha\delta^{-\beta-1}.
\]
Write $\mathcal{T}^k$ for generation $k$ and $\EE_x$ for expectation when the root has type $x$. Iterating the offspring intensity gives, for $k\geq1$ and $x\in[\delta,1]$,
\begin{align*}
\EE_x|\mathcal{T}^k|
&=\int_{x<y_1<\cdots<y_k<1}
\kappa(x,y_1)\kappa(y_1,y_2)\cdots\kappa(y_{k-1},y_k)
\,\textup{d}y_1\cdots\textup{d}y_k\\
&\leq \frac{M_\delta^k(1-x)^k}{k!}
\leq \frac{M_\delta^k}{k!}\longrightarrow0.
\end{align*}
Each particle has a finite number of offspring almost surely. Thus an infinite tree must have a nonempty generation at every depth. Markov's inequality therefore yields
\[
\PP\big(|\mathcal{T}|=\infty\mid U_0>\delta\big)
\leq \PP\big(|\mathcal{T}^k|\geq1\mid U_0>\delta\big)
\leq \frac{M_\delta^k}{k!}\longrightarrow0.
\]
Consequently $\PP(|\mathcal{T}|=\infty)\leq\PP(U_0\leq\delta)=\delta$. Letting $\delta\downarrow0$ proves almost sure finiteness. The exploration coupling now implies tightness of the size of the forward component of a uniformly chosen vertex; in particular,
\[
\lim_{N\to\infty}\PP\big(|C_N(U)|\geq\varepsilon N\big)=0,
\qquad \varepsilon>0,
\]
as claimed; see also \cite[Corollary 2.27]{hofstad_book2}.
\end{proof}

\subsection{Concentration bounds}
We begin by stating a straightforward concentration bound on the maximal degree in $G_N$ which applies not only to our model but to any directed inhomogeneous random graph in which $\kappa$ and $\phi$ are bounded.
\begin{lemma}\label{lem:maxdegree}
	Let $G_N$ denote a random graph in which each arc $(m,n)$ is included independently with probability $p_{mn}^{(N)}.$ Suppose that $p_{mn}^{(N)}\leq A/N$ uniformly in $m,n$ for all sufficiently large $N$ and some $A<\infty$, then the maximal total degree $\Delta_N$ in $G_N$ satisfies
	\begin{equation}\label{eq:maxdeg}
		\PP(\Delta_N> 2A+\lambda)\leq N \textup{e}^{-\frac{\lambda^2}{4A+2\lambda/3}}
	\end{equation}
	for any $N$. In particular, it follows that almost surely
	\[
	\Delta_N\in O(\log N).
	\]
\end{lemma}
\begin{proof}
	Let $D_n$ denote the total degree of $n\in V(G_N)$. By assumption, $D_n$ is dominated by a $\operatorname{Bin}(2N,A/N)$ random variable, hence
	\[
	\PP(D_n\geq 2A+\lambda)\leq \textup{e}^{-\frac{\lambda^2}{4A+\frac{2\lambda}{3}}},\quad n=1,\dots,N,
	\]
	using a classical concentration inequality such as \cite[Theorem 3.2]{ChungLu06}. Summing over $n$ yields \eqref{eq:maxdeg} and we note that choosing $\lambda=\lambda(N)=C \log N$ implies that
	\[N \textup{e}^{-\frac{\lambda^2}{4A+2\lambda/3}}<\infty, \]
	if $C$ is sufficiently large, which allows us to conclude the almost sure bound on $\Delta_N$.
\end{proof}

\begin{proof}[Proof of Proposition~\ref{thm:degev}]
	The statement about the maximal degree is an immediate consequence of Lemma~\ref{lem:maxdegree}. Let us now establish the stated bounds on the degree evolutions. The evolution of the indegree $D_n^{\mathsf{in}}(m)$ of a fixed vertex $m$ after the arrival of vertex $k>m$ can be modelled using a counting process that jumps by $1$ upon the arrival of the $j$th vertex after $k$ if this vertex attaches to $m$. Hence, jumps occur with probability
	\begin{equation}\label{eq:attachement_prob_index}
		p_{k+j}(m)=\frac{w_{k+j}(m)}{k+j-1}\wedge 1, \quad j\in\mathbb{N}_0,
	\end{equation}
	for $m$ and $k>m$ fixed. The indegree of vertex $m$ upon the arrival of vertex $n$ is hence given by a  generalized binomial distribution and satisfies
	\begin{align}\label{eq:expected_indegree}
		\mathsf{E}\left[D_n^{\mathsf{in}}(m)\right] & = \sum_{j=m+1}^{n} p_{j}(m) 
	\end{align}
	In particular, we have
	\[
	\mathsf{E}\left[D_n^{\mathsf{in}}(m)\right]\leq \frac{\alpha}{\beta}\left[1-\left(\frac{m}{n-1}\right)^\beta \right], \quad 1\leq m<n,
	\]
	and
	\[
	\mathsf{E}\left[D_n^{\mathsf{in}}(m)\right]\geq \frac{\alpha}{\beta}\left[\left(1-\frac{1}{m+1}\right)^\beta -\left(\frac{m+1}{n}\right)^\beta \right], \quad \lceil\alpha\rceil\leq m<n.
	\]
	The moment generating function of $D_n^{\mathsf{in}}(m)$ is given as
	\begin{equation}\label{eq:generatingfunc}\mathsf{E}\left[\textup{e}^{\theta D_n^{\mathsf{in}}(m)}\right] = \prod_{j=m+1}^{n} (1 + p_j(m)(\textup{e}^{\theta} - 1)), \quad \theta>0.\end{equation}
	A bound on the tail of the indegree distribution can now be derived using Markov's inequality
	\[
	\PP[D_n^{\mathsf{in}}(m) > x] \leq \inf_{\theta>0} \textup{e}^{-\theta x } \prod_{j=m+1}^{n} (1 + p_j(m)(\textup{e}^{\theta} - 1)).
	\]
	We obtain
	\begin{align*} \textup{e}^{-\theta x} \prod_{j=m+1}^{n} \big(1 + p_j(m)(\textup{e}^{\theta} - 1)\big) & \leq \exp\left(-\theta x + (\textup{e}^{\theta}-1)\sum_{j=m+1}^{n}p_j(m)\right)\\ &
		\leq \exp\left(-\theta x + (\textup{e}^{\theta}-1)\frac{\alpha}{\beta}\right),
	\end{align*}
	and the right hand side is minimized by $\theta=\log\frac{x \beta}{\alpha}$. Thus, for any $n>m\geq 1$ we have the uniform upper tail bound
	\[
	\PP[D_n^{\mathsf{in}}(m) > x]\leq \exp\left[-\frac{\alpha}{\beta}-x\left(\log x +\log \frac{\beta}{\alpha}-\frac{\beta}{\alpha} \right)\right], \quad x>\alpha/\beta,
	\]
	which establishes part (i). Turning to part (ii), we recall that the probability that $k$ attaches itself to $i\in[k-1]$ is given by \eqref{eq:attachement_prob_index}. The expected outdegree is hence given by
	\[
	\mathsf{E}\left[D^{\mathsf{out}}(k)\right] = \sum_{i=1}^{k-1} p_k(i) = \frac{\alpha}{\beta+1}+ O(1/k),
	\]
	where the error term can is simply due to approximating the sum by an integral.
\end{proof}
\begin{proof}[Proof of the error bound in Theorem~\ref{thm:typicaldegree}]
	We only provided the detailed calculation for the in-degree distribution, a similar but simpler argument yields the same bound for the out-degree distribution. Let us set, for each $n\in\mathbb{N}$,
	\[
	P_n(k)=\frac{1}{n}\sum_{m=1}^n\mathsf{1}\{D^{\mathsf{in}}_n(m)=k \},\quad k=0,1,2,\dots,
	\]
	and denote the target distribution by $(p(k))_{k=0}^\infty$. We have
	\[
	p(k)= \int_0^1\textup{e}^{-\frac{\alpha}{\beta}(1-s^\beta)}\left(\frac{\alpha}{\beta}(1-s^\beta)\right)^k\frac{1}{k!}\,\textup{d}s,\quad k=0,1,2,\dots,
	\]
	and, applying the triangle inequality,
	\begin{equation}\label{eq:errorsplit}
		\sum_{k=0}^\infty |P_n(k)-p(k)|\leq \sum_{k=0}^\infty |P_n(k)-\EE P_n(k)|+\sum_{k=0}^\infty |\EE P_n(k)-p(k)|, \quad n\in\mathbb{N}.
	\end{equation}
	Note that the first error is probabilistic and that the second error is essentially just a discretisation error, since
	\begin{equation}\label{eq:expPk}
		\EE P_n(k)=\frac{1}{n}\sum_{m=1}^n \textup{e}^{-\sum_{j=m+1}^n (\alpha \frac{m^\beta}{(j-1)^{\beta+1}} )\wedge 1}\frac{\left(\sum_{j=m+1}^n (\alpha \frac{m^\beta}{(j-1)^{\beta+1}} )\wedge 1\right)^k}{k!},\; k\leq n,
	\end{equation}
	which can be obtained e.g.\ by differentiating \eqref{eq:generatingfunc}. We calculate the bound for each error in the following two lemmas.
	\begin{lemma}\label{lem:deterministic}
		We have that
		\begin{equation}\label{eq:errorbddeter}
			\sum_{k=0}^\infty |\EE P_n(k)-p(k)|=O\left(\frac{\log n}{n}\right)
		\end{equation}
	\end{lemma}
	\begin{proof}
		First note, that $(\alpha \frac{m^\beta}{(j-1)^{\beta+1}} )\wedge 1=1$ only if $m\leq j-1 \leq n_0$ and thus the error induced in the calculation of the probability weights by ignoring the truncation is $O(n^{-1})$ and may thus be ignored. A straightforward integral approximation yields
		\[
		\alpha m^\beta\sum_{j=m+1}^n (j-1)^{-\beta-1}=\frac{\alpha}{\beta}\left(1-\left(\frac{m}{n}\right)^\beta\right)+O(m^{-1}),
		\]
		for all sufficiently large $n$. Since $\frac{\textup{d}}{\textup{d}x}\textup{e}^{-x}x^{k}=(k-x)\textup{e}^{-x}x^{k-1}$, we obtain that, for any $C<\infty$, the function $\textup{e}^{-x}x^{k}$ is uniformly Lipschitz continuous in $x$ on $[-C/m,\alpha/\beta+C/m]$ with $m$-independent Lipschitz constant
		\[
		L_k= \ell_1 k \ell_2^{k-1}, \quad k\in\mathbb{N},
		\]
		where $\ell_1,\ell_2$ are independent of $k$. Setting further $L_0=K$ for some sufficiently large $K$, we conclude that there is a $k$-independent constant $C<\infty$, such that
		\begin{align*}
			\Bigg|\frac{1}{n}\sum_{m=1}^n & \textup{e}^{-\sum_{j=m+1}^n \alpha \frac{m^\beta}{(j-1)^{\beta+1}} }\left(\sum_{j=m+1}^n \alpha \frac{m^\beta}{(j-1)^{\beta+1}} \right)^k \\
			& - \frac{1}{n}\sum_{m=1}^n \textup{e}^{-\frac{\alpha}{\beta}\left(1-\left(\frac{m}{n}\right)^\beta\right)}\left(\frac{\alpha}{\beta}\left(1-\left(\frac{m}{n}\right)^\beta\right)\right)^k\Bigg|\leq \frac{C L_k}{n}\sum_{m=1}^n\frac{1}{m}.
		\end{align*}
		To approximate the discrete outer sum by an integral we use that if $f:[a,b]\to\mathbb{R}$ is continuously differentiable, then
		\[
		\left|\frac{1}{n}\sum_{i=1}^{n}f\left(\frac{i(b-a)}{n}\right)-\int_a^b f(s) \,\textup{d}s \right|\leq \frac{b-a}{2 n} \sup_{x\in[a,b]}f'(x),
		\]
		hence the constants $L_k, k\geq 0$, can be used to bound this approximation error as well and we end up with
		\begin{align*}
			\Bigg|\frac{1}{n}\sum_{m=1}^n & \textup{e}^{-\sum_{j=m+1}^n \alpha \frac{m^\beta}{(j-1)^{\beta+1}} }\left(\sum_{j=m+1}^n \alpha \frac{m^\beta}{(j-1)^{\beta+1}} \right)^k \\
			& - \int_{0}^1 \textup{e}^{-\frac{\alpha}{\beta}\left(1-s^\beta\right)}\left(\frac{\alpha}{\beta}\left(1-s^\beta\right)\right)^k \,\textup{d}s\Bigg|\leq \frac{C_1 L_k}{n}\sum_{m=1}^n\frac{1}{m} + \frac{C_2}{n},
		\end{align*}
		where $C_1,C_2$ are finite constants and independent of $n,m$ and $k$. Noting further that $\sum_{k=0}^\infty\frac{L_k}{k!}<\infty$, we arrive at \eqref{eq:errorbddeter}, since
		\[
		\sum_{k=1}^n|\EE P_n(k)-p(k)| + \sum_{k=n}^\infty p_k \leq \frac{(C_1+C_2)}{n}\sum_{m=1}^n\frac{1}{m} \sum_{k=1}^n \frac{L_k}{k!} + \sum_{k=n}^\infty p_k =O(\log n/n),
		\]
		since $\sum_{k=n}^\infty p_k$ decays exponentially in $n$.

	\end{proof}
	\begin{lemma}\label{lem:concentrationbound}
		We have that almost surely
		\begin{equation*}
			\sum_{k=0}^\infty |P_n(k)-\EE P_n(k)|=O\left({\frac{(\log n)^{3/2}}{\sqrt n}}\right).
		\end{equation*}
	\end{lemma}
	\begin{proof}
		Fix $k$ and write $X_i=\mathsf{1}\{D_n(i)=k\}, i=1,\dots,n$. Note that the $X_i,i=1,\dots,n$ are independent and thus another invocation of Chernoff's inequality (see e.g.\ \cite[Thm.3.2]{ChungLu06}) yields that	
		\begin{equation}\label{eq:Chernoffbound}
			\PP\left(\left|\sum_{i=1}^n X_i-\EE X_i\right|>n\lambda\right)\leq 2\exp\left(-\frac{(\lambda n)^2}{2(\sum_{i=1}^n \EE X_i + \frac{\lambda n}{3})} \right)=2\exp\left(-\frac{\lambda^2 n}{2(\EE P_n(k) + \frac{\lambda}{3})} \right),
		\end{equation}
		with $\lambda>0$. Let $E_n$ denote the event that there no vertex $i\in [n]$ with $D_n(i)\geq C\log n$, then for any $\lambda'>0$
		\begin{align*}
			\PP\Big(\sum_{k=0}^\infty |P_n(k)-\EE P_n(k)|\geq \lambda'\Big)& \leq \PP\Big(\Big\{ \sum_{k=0}^\infty |P_n(k)-\EE P_n(k)| \geq \lambda' \Big\}\cap E_n\Big) +\PP(E^{\mathsf{c}}_n)\\
			& \leq \PP\Big(\sum_{k=0}^{C \log n} |P_n(k)-\EE P_n(k)| + \sum_{k= C\log n}^{\infty}\EE P_n(k) \geq \lambda' \Big) +\PP(E^{\mathsf{c}}_n).
		\end{align*}
		Arguing as in the proof of Lemma~\ref{lem:maxdegree}, we may choose $C$ so large that $\sum_{n=1}^\infty\PP(E^{\mathsf{c}}_n)=B<\infty$ and such that $$\sum_{k= C\log n}^{\infty}\EE P_n(k)\leq \frac{1}{n}\EE \#\{i\in[n]:D_n(i)\geq C\log n\}\leq \frac{1}{n}$$
		for all sufficiently large $n$. Observe further, that \[
		\sum_{k=0}^{C \log n} |P_n(k)-\EE P_n(k)|\geq \lambda'-1/n
		\]
		is only possible, if at least one of the terms $|P_n(k)-\EE P_n(k)|$ exceeds $(\lambda'-1/n)/(C\log n)$ and we conclude that
		\[
		\PP\Big(\sum_{k=0}^\infty |P_n(k)-\EE P_n(k)|\geq \lambda'\Big)\leq \PP(E^{\mathsf{c}}_n) + \sum_{k=0}^{C\log n}\PP\Big(|P_n(k)-\EE P_n(k)|\geq \frac{\lambda'-1/n}{C\log n}\Big).
		\]
		Choosing now $\lambda=\lambda(N)= D \sqrt{\log n / n}$ for some sufficiently large $D$ and setting $\lambda'=\lambda'(n)=\lambda C \log n +1/n$ allows us to apply \eqref{eq:Chernoffbound} to obtain
		\[
		\sum_{n=1}^\infty \PP\Big(\sum_{k=0}^\infty |P_n(k)-\EE P_n(k)|\geq R\frac{(\log n)^{3/2}}{\sqrt{n}}\Big)\leq B + C \sum_{n=1}^\infty  \textup{e}^{-\lambda^2 n /(2+2\lambda/3)}\log n<\infty,
		\]
		for some $R<\infty$ which implies that almost surely
		\[
		\sum_{k=0}^\infty |P_n(k)-\EE P_n(k)|=O\left(\frac{\log n^{3/2}}{\sqrt{n}}\right).
		\]
	\end{proof}
	Combining the previous lemmas concludes the proof of Theorem~\ref{thm:typicaldegree}.
\end{proof}

\subsection{Forward component size and weight evolution -- Proof of Theorem~\ref{thm:giant}}
Let $A\subset [n]$, then
\[
w_{n}(A)=\sum_{m\in A } w_{n}(m), \quad n\in\mathbb{N},
\]
the \emph{weight} of $A$ upon arrival of vertex $n+1$ essentially determines how likely it is that $n+1$ connects to a vertex in $A$. Set
\[
\Gamma_n(m)=|C_{n}(m)|,\quad n_0< m \leq n,
\]
i.e.\ $\Gamma_n(m)$ counts how often $m$ has been referenced (directly and indirectly). We have $\Gamma_{m}(m)=0$ and
\[
\Gamma_{n+1}(m)=\begin{cases}
	\Gamma_{n}(m), & \text{ with probability } \prod_{l\in C_n(m)}\left(1-\frac{w_{n}(l)}{n}\right),\\
	\Gamma_{n}(m)+1, & \text{ with probability }1 - \prod_{l\in C_n(m)}\left(1-\frac{w_{n}(l)}{n}\right),
\end{cases} n>m.
\]
Hence transitions of the processes $\Gamma(m)=(\Gamma_n(m))_{n\geq m}$ depend on the weight structure of the forward component of $m$, not only on its present size. In particular, each $\Gamma(m)$ is a Markov process with respect to the filtration generated by the graph sequence $(G_n)_{n\geq m}$ but not with respect to the smaller filtration generated by the sequence $(\Gamma_n(m))_{n\geq m}$ itself. It is therefore more convenient to study the weight processes $W_n(m)=w_n(C_n(m)), n\geq m$, which have the dynamics
\[
W_{n+1}(m)=\begin{cases}
	\left(\frac{n}{n+1}\right)^\beta W_{n}(m), & \text{ with probability } \prod_{k\in C_n(m)}\left(1-\frac{w_n(k)}{n}\right),\\
	\left(\frac{n}{n+1}\right)^\beta W_{n}(m)+\alpha, & \text{ with probability }1 - \prod_{k\in C_n(m)}\left(1-\frac{w_n(k)}{n}\right).
\end{cases}
\]
To eliminate the slightly unwieldy products structure of the transition probability, we formulate a stochastic domination result. We use the notation
$$\Delta a_n= a_{n+1}-a_n, \quad n=n_1,n_1+1,n_1+2,\dots,$$
for the increments of a random or deterministic sequence $(a_n)_{n\geq n_1}$ and further write $\{W_{n+1}\, \uparrow \} $ for the event that $\alpha$ is added to the score process in step $n+1$. The complementary event is denoted by $\{W_{n+1}\, \downarrow \}$ and we extend this notation to all other processes $X=(X_n)$ appearing below for which $\Delta X_n$ can attain precisely two values.
\begin{lemma}\label{lem:dominationlowerbound}
	Let $Y_n(m)=Y_n(m; a, b , y)$ denote the Markov chain given by $Y_m(m)=y$ and
	\[
	\Delta Y_n(m)=\begin{cases} -\frac{b+1}{n+1} Y_n(m) + \frac{a}{n+1}, \; & \text{with probability }\left(1-\textup{e}^{-Y_n(m)}\right),\\
		-\frac{b+1}{n+1} Y_n(m), \; & \text{with probability }\textup{e}^{-Y_n(m)},
	\end{cases} \quad n\geq m.
	\]
	Then $(W_n(m)/n)_{n\geq m}$ stochastically dominates $\big(Y_n(m; \alpha, \beta , W_m(m)/m)\big)_{n\geq m}$.
\end{lemma}
\begin{proof}
	We start by analysing the transition probabilities of $W_n=W_n(m):$
	\begin{align*}
		1-\prod_{k\in C_n(m)}\left(1-\frac{w_n(k)}{n}\right) & = 1-\textup{e}^{\sum_{k\in C_n(m)}\log(1-w_n(k)/n)}\\
		& = 1-\exp\left(-\sum_{k\in C_n(m)}\sum_{j=1}^\infty \frac{\left(\frac{w_n(k)}{n}\right)^j}{j}\right).
	\end{align*}
	Noting that $w_n(k)/n$ is uniformly bounded by $\frac{\alpha}{n}$, we can estimate
	\begin{equation}\label{eq:upperlower}
		\begin{aligned}
			1-\textup{e}^{-\sum_{k\in C_n(m)} \frac{w_n(k)}{n}} & \leq  1-\exp\left(  -\sum_{k\in C_n(m)}\sum_{j=1}^\infty \frac{\left(\frac{w_n(k)}{n}\right)^j}{j}\right)\\	
			& \leq  1-\exp\left(-\sum_{k\in C_n(m)}\left(\frac{w_n(k)}{n}+(1+\varepsilon)\frac{w_n(k)^2}{2n^2}\right)\right),
		\end{aligned}
	\end{equation}
	for any $\varepsilon>0$, if $n$ is sufficiently large. Another Taylor expansion yields
	\begin{equation}\label{eq:Taylor2}
		1-\frac{\beta+1}{n+1}\leq \left(\frac{n}{n+1}\right)^{\beta+1}\leq 1-\frac{\beta+1}{n+1}+\frac{\beta(\beta+1)}{2(n+1)^2},\quad n\geq m.
	\end{equation}
	We conclude that
	\begin{align*}
		\Delta \frac{W_n}{n} & = \left(1-\left(\frac{n}{n+1}\right)^{\beta+1}\right)\frac{W_n}{n} + \frac{\alpha}{n+1}\mathsf{1}\{n+1 \text{ connects to }C_n(m) \}\\
		& \geq -\frac{\beta+1}{n+1}\frac{W_n}{n} + \frac{\alpha}{n+1}\mathsf{1}\{n+1 \text{ connects to }C_n(m) \},
	\end{align*}
	and due to the lower bound in \eqref{eq:upperlower}, the event in the indicator occurs with probability at least $$	1-\textup{e}^{-\sum_{k\in C_n(m)} \frac{w_n(k)}{n}}=1-\textup{e}^{-W_n/n},$$
	given $C_n(m)$, which proves the lemma.
\end{proof}

To obtain a stochastic upper bound is only slightly more tricky.
\begin{lemma}\label{lem:dominationupperbound}
	Let $Y_n(m; a,b,y)$ be defined as in Lemma \ref{lem:dominationlowerbound} and let $a>\alpha, b<\beta$. There exits $n_1$ such that for all $n\geq n_1$, we can couple
	$Y_n=Y_n(n_1; a,b, w)$ and $W_n=W_n(m)$ conditionally on $\frac{a}{\alpha n_1}W_{n_1}=w$ such that
	$$
	Y_n\geq \frac{a}{\alpha n} W_{n}, \quad \text{ for all }n\geq n_1.
	$$
\end{lemma}
\begin{proof}
	We argue by induction. The claim holds by construction at any given initial time $n_1\geq m.$ Now assume that the coupling has been established for times $n_1,\dots,n$ then
	\begin{align*}
		Y_{n+1}=\left(1-\frac{b+1}{n+1}\right)Y_n +\frac{a}{n+1}\mathsf{1}\{Y_{n+1} \; \uparrow \} \geq \left(1-\frac{b+1}{n+1}\right)\frac{a}{\alpha}\frac{W_n}{n}+ \frac{a}{\alpha}\frac{\alpha}{n+1}\mathsf{1}\{Y_{n+1} \; \uparrow \},
	\end{align*}
	where we have used the definition of $Y_n$ and the induction hypothesis. Now note that
	$$ \PP(Y_{n+1} \; \uparrow | W_n, Y_n)=1-\textup{e}^{-Y_n}\geq 1-\textup{e}^{-\frac a{\alpha n} W_n}\geq \PP(W_{n+1} \; \uparrow | W_n, Y_n),$$
	by the induction hypothesis and the upper bound in \eqref{eq:upperlower} if $n\geq n_1$ is sufficiently large. We conclude that $W_{n+1}$ and $Y_{n+1}$ can be coupled such that $\{W_{n+1}\;\uparrow\}\subset \{Y_{n+1}\;\uparrow\}$ and it follows that under this coupling
	\begin{align*}
		Y_{n+1} & \geq \frac{a}{\alpha}\left[\left(1-\frac{b+1}{n+1}\right)\frac{W_n}{n}+ \frac{\alpha}{n+1}\mathsf{1}\{Y_{n+1} \; \uparrow \}\right]\\ &
		\geq  \frac{a}{\alpha}\left[\left(\frac{n}{n+1}\right)^{\beta+1}\frac{W_n}{n}+ \frac{\alpha}{n+1}\mathsf{1}\{W_{n+1} \; \uparrow \}\right]\\
		& = \frac{a}{\alpha} \frac{W_{n+1}}{n+1}
	\end{align*}
	where we have used \eqref{eq:Taylor2} and once more that $n_1$ can be taken large.
\end{proof}
Lemmas \ref{lem:dominationlowerbound} and \ref{lem:dominationupperbound} indicate that the dynamics of $(C_n(m))_{n\geq m}$ can be analysed in the terms of the simple processes $(Y_n(m))_{n\geq m}$ defined in Lemma~ \ref{lem:dominationlowerbound}. To this end fix $a,b>0$ and $Y_m=y$, and define for $n\geq m$
\[
M_n = Y_n-Y_m +  \sum_{j=m}^{n-1}\Bigg\{\frac{b+1}{j+1}Y_{j} -  \frac{a}{j+1}\left(1-\textup{e}^{-Y_{j}}\right)\Bigg\},
\]
and
\begin{align*}
	Z_n = \frac{Y_n}{Y_m}& \exp\Bigg\{\sum_{j=m}^{n-1}\left[\frac{b+1}{j+1} -  \frac{a}{j+1} \frac{1-\textup{e}^{-Y_{j}}}{Y_{j}}\right] \Bigg\}.
\end{align*}
It follows from a straightforward calculation or, alternatively, by invoking \cite[Lemma 2.1,Lemma 2.2]{BrighLucza2012} that $(M_n)_{n\geq m}$ is a martingale and that $(Z_n)_{n\geq m}$ is a super-martingale, respectively, adapted to the filtration $(\mathcal{F}_n)_{n\geq m}$ defined by $\mathcal{F}_n=\sigma(Y_j, m \leq j\leq n)$.
\begin{lemma}\label{lem:subcritical}
	If $0<a<b+1$ then $\lim_{n\to\infty}Y_n=0$ almost surely.
\end{lemma}

\begin{proof}
Since $(Z_n)_{n\geq m}$ is a non-negative super-martingale, $(Z_n)_{n\geq m}$ converges almost surely to some random variable $Z_\infty$ satisfying $\EE Z_\infty \leq \EE Z_m = 1.$ 

	Since $a<b+1$, it thus follows from
	\begin{align*}
		& \exp\Bigg\{\sum_{j=m}^{n-1}\left[\frac{b+1}{j+1} -  \frac{\alpha}{j+1} \frac{1-\textup{e}^{-Y_{j}}}{Y_{j}}\right] \Bigg\}\\ & \geq \exp\Bigg\{\sum_{j=m}^{n-1}\left[\frac{b+1}{j+1} -  \frac{a}{j+1}\right] \Bigg\} = n^{1+b - a  + o(1)},
	\end{align*}
	that $Z_\infty$ can only be finite, if $(Y_n)_{n\geq m}$ converges to $0$. Hence, $\lim_{n\to\infty }Y_n=0$ almost surely.
	
\end{proof}
Let us now focus on the case $a>b+1$. Performing a Taylor expansion, we see that
\begin{equation}\label{eq:TaylorY}
	a (1-\textup{e}^{-y})-(1+b)y = \left(a \textup{e}^{-y_\ast}-(1+ b)\right)(y-y_\ast)-\frac{a}{2}\xi(y)\textup{e}^{-\xi(y)}(y-y_\ast)^2,
\end{equation}
where $\xi(y)$ is some value between $y_\ast$ and $y$ and $y_\ast=y_\ast(a,b)$ is the unique positive solution to
\begin{equation}\label{eq:fixointab}
	1-\textup{e}^{-y}= \frac{1+\beta}{\alpha}y.
\end{equation}
Using the Martingale $(M_n)_{n\geq m}$, we can represent $Y_n$ as
\begin{equation}\label{eq:Ypose}
	\begin{aligned}
		Y_m & = y>0\\
		\Delta Y_n & = \frac{a}{n+1}\left(1-\textup{e}^{-Y_n}\right)- \frac{b+1}{n+1} Y_n + \Delta M_n, \quad n\geq m.
	\end{aligned}
\end{equation}
In view of \eqref{eq:TaylorY} and \eqref{eq:fixointab}, the recursion \eqref{eq:Ypose} is a stochastic approximation of the stable point $y_\ast.$ We recall the following well-known result about stochastic approximations:
\begin{lemma}[{Special case of \cite[Lemma 2.6]{Peman2007}}]\label{lem:stochasticapp}
	Suppose that $(X_n)_{n\in\mathbb{N}}$ is a stochastic process adapted to a filtration $(\mathcal{G}_n)_{n\in\mathbb{N}}$ satisfying the recurrence relation
	\[
	\Delta X_n=c_n\left(f(X_{n})+\xi_{n+1}\right),
	\]
	where $(c_n)_{n\in\mathbb{N}}$ is a deterministic sequence, $(\xi_n)_{n\in\mathbb{N}}$ is a $(\mathcal{G}_n)_{n\in\mathbb{N}}$-adapted process and $f$ is a bounded function satisfying
	\begin{itemize}
		\item $\sum_{n\in\mathbb{N}}c_n=\infty,\; \sum_{n\in\mathbb{N}}c_n^2<\infty$,
		\item $\EE[\xi_{n+1}|\mathcal{G}_n]=0, \; \EE[\xi_{n+1}^2|\mathcal{G}_n]\leq B<\infty$ for all $n\in\mathbb{N}$,
		\item either $f((l_0,r_0))\subset (-\infty,-\delta)$ or $f((l_0,r_0))\subset (\delta,\infty)$ for some $\delta>0$ and some open interval $(l_0,r_0)$.
	\end{itemize}
	Then we have for any interval $[l,r]\subset(l_0,r_0)$ that
	\[
	\PP(X_n \text{ visits }[l,r] \text{ only finitely many times})=1.
	\]
\end{lemma}
We have now all the ingredients to complete the proof of Theorem~\ref{thm:giant}.
\begin{proof}[Proof of {Theorem \ref{thm:giant}}]
	Let $\Gamma_n=|C_n(m)|$, $n\geq m$ and $W_n=w_n(C_n(m))$. We estimate
	\[
	\frac{\alpha}{n^{\beta}}\sum_{k=m}^{m+\Gamma_n-1} k^\beta \leq W_n\leq \frac{\alpha}{n^{\beta}}\sum_{k=n-\Gamma_n}^{n} k^\beta, \quad n\geq m.
	\]
	Approximating the sums by integrals yields
	\[
	\frac{\alpha}{\beta+1}n\left(\frac{\Gamma_n}{n} \right)^{\beta+1} + O(n^{-\beta}) \leq W_n\leq \frac{\alpha}{\beta+1} n\left( 1-\left(1-\frac{\Gamma_n}{n} \right)^{\beta+1}\right)+ O(n^{-\beta}) , \quad n\geq m,
	\]
	from which we conclude that $\lim_{n\to\infty}\frac{\Gamma_n}{n}=Z$ exists if and only if $\lim_{n\to\infty}\frac{W_n}{n}=w_\infty$ exists.\\
	\noindent\textit{The subcritical case }$\alpha<\beta+1:$ By Lemma~\ref{lem:dominationupperbound}, $W_n/n$ is eventually dominated by $Y_n(n_1, \alpha',\beta', y)$ for some $y$, if $\alpha<\alpha'<\beta'+1<\beta+1$. Since $\lim_{n\to\infty} Y_n=0$ almost surely by Lemma~\ref{lem:subcritical}, we conclude that
	$$
	\PP(Z=w_\infty=0)=1.
	$$
	\noindent\textit{The supercritical case }$\alpha>\beta+1:$ We first employ Lemma~\ref{lem:dominationlowerbound} to dominate $w_\infty$ from below by the random limit
	\[
	Y_{\infty}=Y_\infty(\alpha,\beta)=\lim_{n\to\infty}Y_n(m; \alpha,\beta, W_m/m),
	\]
	which we show to exist now. Rewriting \eqref{eq:Ypose}, we obtain
	\[
	\Delta Y_n =\frac{1}{n+1}\left(\alpha (1-\textup{e}^{-Y_n})-(\beta+1)Y_n + (n+1)\Delta M_n \right),\quad n\geq m.
	\]
	We wish to apply Lemma~\ref{lem:stochasticapp} with $c_n=(n+1)^{-1}, \xi_{n+1}=(n+1)\Delta M_n$ and $$f(x)= \mathsf{1}_{[0,K]}(x)\left(\alpha (1-\textup{e}^{-x})-(\beta+1)x\right), \quad x\in\mathbb{R}.$$
	Let us first check that the condition on $\xi_n$ is satisfied. Clearly, $\EE [\xi_{n+1}| Y_n]=0$, since $\Delta M_n$ is a martingale difference. To see that $\EE [\xi_{n+1}^2| Y_n]$ is bounded, we simply note that
	$$
	|\Delta M_n| \leq |\Delta Y_n|+ (1+Y_n)\frac{\alpha+\beta+1}{n+1}\leq \frac{C}{n+1}
	$$
	for some constant $C$ by definition of $Y_n$ and the fact that $Y_n$ is bounded, hence $\xi_n$ is bounded. Let us now analyse the recursion function $f$. Note that the truncation of $f$ is necessary du to the assumption of boundedness in Lemma~\ref{lem:stochasticapp}, but this is no problem, since clearly $Y_n>0$ for all $n$ and $Y_n\leq W_n/n$ due to the coupling and $W_n/n$ is easily seen to be bounded by some deterministic value $K>\frac{\alpha}{\beta+1}$ (mind that $W_n/n\leq \frac{\alpha}{\beta+1}+o(1)$.) Note that $f$ is continuous on $[0,K]$ with $f(x)=0$ for $x\in[0,K]$ if and only if $x\in\{0,y_\ast\}$, where $y_{\ast}=y_{\ast}(\alpha,\beta)$ is given in \eqref{eq:fixointab}. Note further that $f(x)>0$ if $x\in(0,y_\ast)$ and $f(x)<0$ if $x\in (y_\ast,K)$. Continuity of $f$ now implies that for any sufficiently small $\varepsilon_0>0$ there exits a $\delta_0>0$ such that
	\[
	f(x)>\delta_0 \text{ for all }x\in (\varepsilon_0, y_{\ast}-\varepsilon) \text{ and } f(x)<-\delta_0 \text{ for all }x\in \left(y_{\ast}+\varepsilon_0, \frac{\alpha}{\beta+1}+\varepsilon_0\right),
	\]
	where we note that $y_{\ast}<\frac{\alpha}{\beta+1}$ for any choice of $\alpha>\beta+1$. By Lemma~\ref{lem:stochasticapp}, we deduce that any closed bounded interval not including either $0$ or $y_\ast$ is visited only finitely often and hence $(Y_n)_{n\geq m}$ converges almost surely to a random variable $Y_\infty$ with distribution
	$$
	(1-p)\delta_0 + p \delta_{y_\ast}, \quad \text{for some } p\in[0,1].
	$$
	We say $(Y_n)_{n\geq m}$ \emph{survives} if $Y_\infty=y_\ast$ and that it \emph{dies out} if $Y_\infty=0$. We are going to show now that $p>0$, i.e.\ $(Y_n)_{n\geq m}$ survives with positive probability. Observe that we have
	\[
	\Delta \EE [Y_n|\mathcal{F}_m] = \frac{\alpha}{n+1}\EE [1-\textup{e}^{-Y_n}|\mathcal{F}_m] - \frac{\beta+1}{n+1}\EE [Y_n|\mathcal{F}_m] , \quad n\geq k.
	\]
	Note that if $(Y_n)_{n\geq m} \text{ dies out}$ a.s., then we have that $\lim_{k\to\infty}\sup_{n\geq k}\EE [Y_n|\mathcal{F}_k]=0$ almost surely. Since $\EE [1-\textup{e}^{-Y_n}|\mathcal{F}_k]=\EE [Y_n|\mathcal{F}_k] + O(\EE [Y_n|\mathcal{F}_k]^2)$ as $\EE [Y_n|\mathcal{F}_k]$ vanishes, it follows that for every $\varepsilon\in (0,\alpha-\beta-1)$ there exists some $K_0$ such that
	\[
	\Delta \EE [Y_n|\mathcal{F}_k] \geq \frac{\alpha-\beta-1-\varepsilon}{n+1}\EE [Y_n|\mathcal{F}_k]   , \quad n\geq k\geq K_0.
	\]
	Denoting $\zeta=\alpha-\beta-1-\varepsilon>0$, and setting $Y_{k}=:a(k)>0$, it follows that $(\EE [Y_n|\mathcal{F}_k])_{n\geq k}$ dominates the unique solution of the difference equation
	\[
	a_{k}=a(k), \quad \Delta a_n=\frac{\zeta}{n+1}a_n,\quad n\geq k,
	\]
	which satisfies $\lim_{n\to\infty}a_n=\infty$ for any initial condition $a(k)>0$. This is a contradiction to $\lim_{k\to\infty}\sup_{n\geq k}\EE [Y_n|\mathcal{F}_k]=0$ and we conclude that $$
	p=\PP((Y_n)_{n\geq m} \text{ survives})>0.
	$$
	Combining Lemmas~\ref{lem:dominationlowerbound} and~\ref{lem:dominationupperbound}, we obtain that
	$$
	y_{\ast}(\alpha,\beta)\leq w_\infty \leq y_\ast(\alpha+\varepsilon,\beta-\varepsilon)
	$$
	for all sufficiently small $\varepsilon$ and since the root $y_\ast(a,b)$ of \eqref{eq:fixointab} depends continuously on $a$ and $b$, we conclude that $w_\infty=\lim_{n\to\infty}\frac{W_n}{n}=y_\ast(\alpha,\beta)>0$ almost surely on the event of survival. It follows that, almost surely on survival, $\lim_{n\to\infty}\frac{\Gamma_n}{n}=:\gamma>0$. Let $q_n$ denote the probability that a uniformly chosen vertex in $G_n$ is contained in $C_n(m)$, then $\lim_{n\to\infty}q_n=\gamma\PP(\lim_{n\to\infty}\Gamma_n/n>0)$ and, for any $m\geq n$,
	\[
	q_n = \frac{1}{n}+\frac{1}{n}\sum_{k=m}^{n-1} \PP(k+1\in C_n(m))=\frac{1}{n}+ \frac{1}{n}\EE\left(\sum_{k=m}^{n-1} \EE [\mathsf{1}\{k+1\in C_{k+1}(m)\}|G_{k}]\right).
	\]
	Note that $\EE [\mathsf{1}\{k+1\in C_{k+1}(m)\}|G_{k}]=1-\textup{e}^{-W_k/k}$ and that for any $\varepsilon\in(0,1)$
	\[
	\sum_{k=\lceil\varepsilon n\rceil-1}^{n-1}\EE(1-\textup{e}^{-W_k/k})\leq \sum_{k=m}^{n-1}\EE(1-\textup{e}^{-W_k/k})\leq \sum_{k=\lceil\varepsilon n\rceil-1}^{n-1}\EE(1-\textup{e}^{-W_k/k}) +\varepsilon n +1,
	\]
	if $n$ is sufficiently large. We infer that $\gamma= \frac{1+\beta}{\alpha} y_\ast(\alpha,\beta)$ and this concludes the proof of the critical case by establishing a one-to-one relation between $Y_\infty$ and the random variable $Z=Z(m)$ as defined in the theorem. \\
	\noindent\textit{The critical case }$\alpha=\beta+1:$ From the supercritical case we infer that almost surely $\gamma=0$ in this case, since $\gamma$ is a monotone function of $\alpha$ and $\lim_{\alpha\downarrow \beta+1} y_\ast(\alpha,\beta)=0$.
\end{proof}

\begin{remark} \begin{enumerate}[(a)]
		\item The original version of Lemma \ref{lem:stochasticapp} in \cite{Peman2007} also allows to include an error term $R_n$ in the recursion as long as $\sum_n c_n R_n<\infty$. In principle one can apply this version of the stochastic approximation result directly to $W_n/n$. We have chosen to include the intermediate step via $Y_n$ because it breaks up the error estimates into several parts and makes the argument easier to follow, since $(Y_n)_{n\geq m}$ has a rather simple form.
		\item Note that the subcritical result also follows from the supercritical/critical case by monotonicity, but we gave a stand-alone argument since the supercritical case is more involved than the direct proof. Unfortunately, this simple argument does not carry over to the critical case.
	\end{enumerate}
\end{remark}

\bibliographystyle{amsplain}
\bibliography{references}
\end{document}